\documentclass[12pt]{amsart}

\usepackage[margin=1in]{geometry}
\usepackage{graphicx}
\usepackage{amsmath}
\usepackage{amssymb}
\usepackage{eucal}
\usepackage{pxfonts}
\usepackage{mathrsfs} 
 \usepackage{mathptmx, tikz}

\newtheorem{theorem}{Theorem}[section]
\newtheorem{lemma}[theorem]{Lemma}
\newtheorem{proposition}[theorem]{Proposition}
\newtheorem{corollary}[theorem]{Corollary}
\newtheorem{definition}[theorem]{Definition}

\newtheorem{remark}[theorem]{Remark}

\theoremstyle{remark}
\newtheorem{example}{Example}[section]

\numberwithin{equation}{section}

 \DeclareSymbolFont{largesymbols}{OMX}{yhex}{m}{n}
  \DeclareMathAccent{\widehat}{\mathord}{largesymbols}{"62}

\DeclareMathOperator{\sat}{sat}

\DeclareMathOperator{\Br}{Br}
\DeclareMathOperator{\Cons}{Cons}
\DeclareMathOperator{\Poly}{Poly}
\DeclareMathOperator{\cl}{cl}

\newcommand{\C}{\CMcal{C}}
\newcommand{\RR}{\mathbf{R}}

\newcommand{\HH}{\CMcal{H}}
\newcommand{\PP}{\mathbf{P}}
\newcommand{\N}{\mathbf{N}}
\newcommand{\E}{\CMcal{E}}

\newcommand{\F}{\CMcal{F}}

\newcommand{\EE}{\mathbf{E}}
\newcommand{\Z}{\mathbf{Z}}
\newcommand{\cals}{\CMcal{S}}

\newcommand{\proj}{\text{proj}}

\newcommand{\Q}{Q}

\newcommand{\A}{A}
\title{Small-time asymptotics for hypoelliptic diffusions }

\author{Juraj F\"{o}ldes}
\author{David P. Herzog}

\begin{document}
\maketitle

\begin{abstract}
An inductive procedure is developed to calculate the asymptotic behavior at time zero of a diffusion with polynomial drift and degenerate, additive noise.  The procedure gives rise to two different rescalings of the process; namely, a functional law of the iterated logarithm rescaling and a distributional rescaling.  The limiting behavior of these rescalings is studied, resulting in two related control problems which are solved in nontrivial examples using methods from geometric control theory.  The control information from these problems gives rise to a practical criteria for points to be regular on the boundary of a domain in $\RR^n$ for such diffusions.     
\end{abstract}

\section{Introduction}
The goal of this work is to explore the time-zero behavior of the following class of stochastic differential equations (SDEs) on $\RR^n$
\begin{align}
\label{eqn:SDEp}
\begin{cases}
d x_t=  P(x_t) \,dt + \sigma \, dB_t\\
x_0 = 0 \,,
\end{cases}
\end{align}
where $P: \RR^n\rightarrow \RR^n$ is a polynomial vector field, $\sigma=\text{diag}(\sigma^1, \sigma^2,\ldots, \sigma^n) $ is a constant, diagonal $n\times n$ matrix with $\sigma^i \geq 0$, and $B_t$ is a standard $n$-dimensional Brownian motion defined on a filtered probability space $(\Omega, \F,  (\F_t)_{t\geq 0},\PP, \EE)$.  Importantly,  some or many of the diagonal elements of $\sigma$ may be zero, so that the noise in~\eqref{eqn:SDEp} is degenerate.

Our study of~\eqref{eqn:SDEp} at time zero is partly motivated by the \emph{regular point problem}.  More specifically,  
in this paper we provide criteria for identifying regular points on the boundary of a domain in $\RR^n$ for diffusions with degenerate noise of the form~\eqref{eqn:SDEp}.  The regular point problem often arises in the context of 
partial differential equations, 
when one investigates well-posedness of the Dirichlet and Poisson problems on a domain $\mathscr{O}\subset \RR^n$ with the differential operator being  the inifinitesimal generator of the diffusion \eqref{eqn:SDEp},  see~\cite{FH_23, OK_13}.  In the context of~\eqref{eqn:SDEp},  assume without loss of generality $0\in \partial \mathscr{O}$ (otherwise change coordinates) and denote
\begin{align}
\xi= \inf\{ t >0 \, : \, x_t \in \mathscr{O}\}. 
\end{align}
If $\PP\{ \xi =0\} =1$,  then we call $0\in \RR^n$ \emph{regular} for $(x_t, \mathscr{O})$,   otherwise, $0\in \RR^n$ is called \emph{irregular} for $(x_t, \mathscr{O})$.  Note that, because the event $\{ \xi =0\}$ belongs to the germ sigma-field $\bigcap_{t>0} \F_t$, Blumenthal's zero-one law implies $\PP\{ \xi =0 \} \in \{ 0,1\}$,  see \cite{Blumenthal1957}.

When the diffusion~\eqref{eqn:SDEp} is non-degenerate;  that is, when $\sigma^i \neq 0$ for all $i$, then there is the well-known \emph{cone condition} ensuring $0\in \RR^n$ is regular for $(x_t, \mathscr{O})$.  More precisely, let $B_r(0)$ denote the open ball of radius $r>0$ centered at $0\in\RR^n$.  If there is a basis $\{v_1, v_2, \ldots, v_n\}$ of $\RR^n$ such that 
\begin{align*}
\text{Cone}(v_1, \ldots, v_n): = \{ \lambda_1 v_1 + \cdots + \lambda_n v_n \, : \, \lambda_i >0 \} 
\end{align*}     
satisfies 
\begin{align*}
\text{Cone}(v_1, \ldots, v_n) \cap B_\rho(0) \subset \mathscr{O}
\end{align*}
 for some $\rho >0$, then $0\in \RR^n$ is regular for $(x_t, \mathscr{O})$.  Intuitively, this is true because the trajectories of $\sigma B_t$ dominate the time-zero behavior of~\eqref{eqn:SDEp} in terms of scale, i.e. $\sqrt{t}$ versus $t$, and the process $\sigma B_t$ enters $\mathscr{O}$ instantaneously as long as there is ``sufficient space", as determined by the cone condition~\cite{BG_07,CFH_22, PS_72}. 
 
 On the other hand, when one or more of the diffusion coefficients $\sigma^i$ vanish, much less is known about regular points,  and results are typically restricted to specific examples.  In the pair of papers~\cite{Lac_97ii, Lac_97i}, the so-called \emph{Iterated Kolmogorov} diffusion is studied, which is a special case of~\eqref{eqn:SDEp} with $\sigma^1 >0$, $\sigma^2=\cdots =\sigma^n=0$ and $P(x)=(0, x^1, x^2, \ldots, x^{n-1})$.  In~\cite{Lac_97i}, a necessary and sufficient condition for $0\in \partial \mathscr{O}$ to be regular for the Iterated Kolmogorov diffusion is given for a general class of boundaries $\partial \mathscr{O}$ and the paper~\cite{Lac_97ii} deduces associated laws of the iterated logarithm.  Under conditions on the fundamental solution,  in~\cite{Kog_17} a general class of hypoelliptic diffusions is studied using methods from PDE and a Wiener, cone-type criterion is deduced for existence and uniqueness of generalized solutions (in the Perron-Wiener sense)  of the associated Dirichlet problem.  One of the main applications in \cite{Kog_17} is a criteria for regular points for linear, hypoelliptic diffusions with additive noise~\cite[Theorem 6.2]{Kog_17}, which is also a special case of~\eqref{eqn:SDEp}.  For earlier related results, we refer to~\cite{GS_90, LTU_17, NS_87} which establish Wiener-type criteria for regular points of \emph{strongly hypoelliptic} diffusions.  Strongly hypoelliptic diffusions are different structurally than hypoelliptic diffusions of the form~\eqref{eqn:SDEp},  since only the noise vector fields,  along with their iterated commutators are assumed to span the tangent space at all points in space.  One of the goals of the present paper is to generalize the strongly hypoelliptic setting and produce results for~\eqref{eqn:SDEp}.

Different from the uniformly elliptic setting; that is, when $\sigma^i \neq 0$ for all $i$, the process $\sigma B_t$ is no longer dominant at time zero in every direction.  In particular, one or more directions in the equation are at a smaller scale than $\sqrt{t}$.  Thus the first issue present in the regular point problem is the identification of the correct scale for the process $x_t$ in every direction at time zero.  We address this issue for equations of the form~\eqref{eqn:SDEp} which are \emph{noise propagating} (see Definition~\ref{def:dimQsig} below) at two scales; namely, a functional law of the iterated logarithm (LIL) scale and a distributional scale.  We describe a constructive inductive procedure that produces the right scale in each component in these two senses.  The procedure starts from the directions in which noise is explicitly present ($\sigma_j \neq 0$) and whose scale is known.  These scales are then propagated through the equation to determine the correct scales in the remaining directions.  It is important to note that the notion of \emph{noise propagating} appears weaker than the notion of hypoellipticity as in~\cite{Hor_67}, and there are simple, concrete examples of diffusions~\eqref{eqn:SDEp} which are noise propagating that are not hypoelliptic (see Example~\ref{ex:NPNH}).  However, we do not have proof of this general assertion.  Intuitively, noise propagating means that the noise implicitly accentuates all directions in space, but the resulting diffusion can still live on a lower-dimensional manifold in $\RR^n$.  The inductive procedures to compute scales are presented in Section~\ref{sec:IL} and Section~\ref{sec:lawscale}, where the notion of \emph{scalings} is introduced and discussed in Section~\ref{sec:scales}.  

As a consequence of the identification of scales for noise propagating diffusions we obtain two notions time-zero behavior.  Namely, a functional LIL at time zero, which, like the LIL for Brownian motion, gives the very precise ``asymptotic windows" in path space for the rescaled process using this particular scaling.  On the other hand, we can also compute the correct scale in terms of distribution on path space.  The latter rescaling is less refined than the functional LIL scaling.  However, in terms of the regular point problem, the resulting control problem which determines which points are visited by the limit of the rescaled process is arguably easier to solve in the distributional case than the LIL scaling.  This is because methods from geometric control theory are now accessible~\cite{AS_04, AS_06, GHHM_18, HM_15, Jur_97, JK_85, NM_24}, whereas in the case of the functional LIL the controls must be bounded.  Nevertheless, due to the related structure of the two rescalings, control information from the distributional control problem can sometimes be transferred to control information for LIL scaling.  This connection is explored in Section~\ref{sec:control} and Section~\ref{sec:gcontrol}.  In particular, in addition to discussing the relationships between the two control problems and how to solve them, we also give a mostly self-contained introduction to various control methods in  Sections~\ref{sec:control} and~\ref{sec:gcontrol}  .

Lastly, in Section~\ref{sec:regularpt}, we establish a criteria for $0\in \RR^n$ to be regular for $(x_t, \mathscr{O})$ for various domains $\mathscr{O}\subset \RR^n$ with $0\in \partial \mathscr{O}$.  Intuitively,  one first computes the distributional scaling inductively using the already introduced procedure,  under the hypothesis that $x_t$ is noise propagating.  Then, one solves the resulting control problem associated to the rescaled process using the methods from geometric control theory discussed in Section~\ref{sec:gcontrol}.  Provided the boundary $\partial \mathscr{O}$ is well-behaved under the scaling to allow sufficient space, and the control trajectories can access this space, it follows  that $0\in \mathscr{O}$ is regular for $(x_t, \mathscr{O})$.

We begin in the next section by introducing notation and terminology used throughout paper.

\section{Notation and Terminology}
\label{sec:NotTerm}

Due to the possibility of finite-time explosion in~\eqref{eqn:SDEp}, we define stopping times $\tau_k$, $k\in \N$, and $\tau_\infty$ by  
\begin{align}
\tau_k= \inf\{ t\geq 0 \, : \, |x_t| \geq k\} \qquad \text{ and } \qquad \tau_\infty= \lim_{k\rightarrow \infty} \tau_k.    
\end{align}
It follows that the (pathwise) solution $x_t$ is defined and unique for all finite times $t< \tau_\infty$, $\PP$-almost surely.  If $\PP \{ \tau_\infty=\infty\}=1$, then we say that $x_t$ is \emph{non-explosive}.   Otherwise, we fix a death state $\Delta\notin \RR^n$ and endow $\RR^n \cup \{ \Delta \}$ with the same topology as the one point compactification; that is, $U$ is open in $\RR^n\cup \{ \Delta \}$ if either $U$ is open in $\RR^n$ or $U = \{ \Delta \} \cup (\RR^n \setminus C)$ for some compact $C\subset \RR^n$.  Then we extend the local solution $x_t$ to all finite times $t\geq0$ by setting $x_t = \Delta$ for $t\geq \tau_\infty$.  In particular, for almost every $\omega \in \Omega$, the mapping $t\mapsto x_t(\omega):[0,1]\rightarrow \RR^n \cup\{\Delta\}$ belongs to the space of \emph{explosive trajectories} $\E$ on $[0,1]$; that is, it belongs to the set $\E$ of continuous mappings $f:[0,1]\rightarrow \RR^n \cup \{ \Delta \}$ such that if $f_{t_0}=\Delta$ for some $t_0 \in [0,1]$ then $f_t = \Delta$ for all $t\geq t_0$ .  For any $g\in \E$, we let 
\begin{align}
\tau(g):= \inf\{ t\in[0,1]\, : \, g_t=\Delta\}
\end{align}
 with the understanding that $\inf \emptyset :=\infty$.  In particular, $\tau(x_\cdot(\omega))= \tau_\infty(\omega) \wedge 1$ for $\PP$-almost every $\omega$.  We let $\C$ denote the space of continuous mappings $f:[0,1]\rightarrow \RR^n$.  In particular, $\C \subset \E$.  We also make use of the spaces
 \begin{align}
\label{eqn:Hdef}
\HH&= \{ h\in \C \, : \, h_0=0, \, \dot{h}\in L^2([0,1]; \RR^n)\} \quad \text{ and } \quad \HH_\alpha = \{ f\in \HH \, : \, \tfrac{1}{2}\textstyle{\int_0^1 |\dot{f}_s|^2 \, ds} \leq \alpha^2 \}, 
\end{align} 
 where $\dot{h} = dh/dt$ and $\alpha >0$.

 For any $t\in (0,1]$, we define $d_t:\E\times \E\rightarrow [0, \infty]$ by 
 \begin{align}
 \label{eqn:distances}
 d_t(f,g):= \begin{cases}
 \sup_{s\in[0,t]} |f_s-g_s| & \text{ if } \tau(f) \wedge \tau(g)>t\\
 \infty & \text{ otherwise}
 \end{cases}.
 \end{align} 
Observe that $d_t(f,g)$ agrees with the usual sup norm on $[0, t]$ provided both $f$ and $g$ have yet to hit the death state $\Delta$ by time $t$.   For simplicity, we write $d=d_1$.  It is not hard to check that $d_t$ satisfies the triangle inequality.  However, $d_t(f,f) = \infty$ for any $f\in \E$ with $f_{t/2}=\Delta$, and therefore $d_t$ is not a metric on $\E$.  For any $A\subset \E$ and any $f\in \E$, we let
\begin{align}
\label{eqn:distances2}
d_t(f, A) = \inf\{ d_t(f,g) \, : \, g \in A \}. 
\end{align}
 
 Throughout the paper, if $x\in \RR^n$, then we write $x=(x^1, x^2, \ldots, x^n)$ where $x^i\in \RR$.  The $k$th power of $x^i \in \RR$ is denoted by $(x^i)^k$ to avoid confusion.  The set $\{e_1, e_2, \ldots, e_n \}$ denotes the standard orthonormal basis of $\RR^n$, and we use $\cdot$ to denote the usual dot product of vectors.  The notation $B_r(x)$ denotes the open ball centered at $x\in \RR^n$ with radius $r>0$.  The index set $I= \{1,2,\ldots, n\}$ is used frequently below. For any subset $J \subset I$ and $x\in \RR^n$, we let $\pi_{J}(x)\in \RR^n$ denote the projection onto the span of $\{ e_j \}_{j\in J}$; that is, 
\begin{align}
\label{eqn:projection1}
\pi_J(x) = \sum_{j \in J} (x\cdot e_j)  e_j  = \sum_{j\in J} x^j e_j .
\end{align}  
We also use the notation $\proj_j: \RR^2\rightarrow \RR$, $j=1,2$, to denote the projection onto the $j$th coordinate; that is, 
\begin{align}
\label{eqn:projection2}
\proj_1(x^1,x^2)=x^1 \quad \text{ and } \quad \proj_2(x^1,x^2)=x^2.
\end{align}  
The notation $\cl (A)$ denotes the closure of a set $A$, where we indicate the precise topology if it is not clear from context.  The set of constant vector fields on $\RR^n$ is denoted by $\Cons(\RR^n)$, while the set of polynomial vector fields on $\RR^n$ is denoted by $\Poly(\RR^n)$.  Below, we often express $Z\in \Poly(\RR^n)$ in coordinates by 
\begin{align}
\label{eqn:vfgeneric}
Z= \sum_{j=1}^n Z^j(x) \frac{\partial}{\partial x^j}
\end{align}  
for some polynomial functions $Z^1, \ldots, Z^n$.  
For any $x\in \RR^n$ and $Z\in \Poly(\RR^n)$, we let $Z(x)$ denote the point $Z(x)= (Z^1(x), \ldots, Z^n(x))\in \RR^n$, where $Z^1, \ldots, Z^n$ are as in~\eqref{eqn:vfgeneric}.

For any vector field $R\in \Poly(\RR^n)$ and $x\in \RR^n$, we let $\varphi_t(R)x$ denote the maximally defined (on $[0, \infty)$) solution of the ordinary differential equation 
\begin{align}
\label{eqn:flow1}
\begin{cases}
\dot{x}_t = R(x_t) \\
x_0=x.
\end{cases}
\end{align}
Note that since $R$ is a polynomial vector field, it is locally Lipschitz, so the solution~\eqref{eqn:flow1} exists locally in time depending on the initial condition $x\in \RR^n$.  For any $R\in \Poly(\RR^n)$, $x\in \RR^n$ and $f\in \HH$, we denote by $\varphi_t(R, f)x$ the maximally defined (on $[0,1]$) solution of 
\begin{align}
\label{eqn:flow2}
\begin{cases}
\dot{x}_t = R(x_t) + \sigma \dot{f} \\
x_0=x.
\end{cases}
\end{align}
For similar reasons, $\varphi_t(R,f)x$ also exists locally in time.

\section{Algebraic Structure and Scalings}
\label{sec:scales}
There are multiple ways to rescale equation~\eqref{eqn:SDEp} to capture the behavior of the solution at time $t=0$.  We consider two in this paper. While one keeps track of only power asymptotics, the other additionally keeps track of logarithmic corrections.   In this section, we introduce a pair scaling that applies to both.

Define
\begin{align}
\label{def:scalings}
\cals= \tfrac{1}{2}\Z_{\geq 0}\times \tfrac{1}{2}\Z_{\geq 0}.  
\end{align}
We call points in the set $\cals$ \emph{scalings}.  We define an ordering on $\cals$ as follows: if $a,b \in \cals$, then we write $a \preceq b$ if $a^1 < b^1$ or if $a^1 =b^1$ and $a^2 \geq b^2$.  If $a \preceq b$ and $a \neq b$, we write $a \prec b$.  Note that $\preceq$ is a total ordering on $\cals$.   It is convenient to augment the set $\cals$ with the \emph{infinite scaling} $a_\infty=\infty$ which satisfies $b \prec a_\infty$ for any $b\in \cals$. We introduce a sum $+$ on  
\begin{align}
\label{def:scalings2}
\cals_\infty= \cals\cup \{ \infty\} 
\end{align}
by 
\begin{align}
a+ b:= 
\begin{cases}
(a^1+ b^1, a^2+ b^2 ) & \text{ if } a, b \in \cals \\
\infty & \text{ otherwise}
\end{cases}  .
\end{align}
For $a \in \cals_\infty$ and $\ell\in \Z_{\geq 0}$, we also define \begin{align}
\ell a:=\begin{cases}
(\ell a^1, \ell a^2) & \text{ if } a \in \cals \\
\infty & \text{ if } a = \infty, \, \ell \geq 1 \\
(0,0) & \text{ if } \ell =0  
\end{cases}.
\end{align}

For any monomial $m: \RR^k\rightarrow \RR$ given by 
\begin{align}
m(x)=m(x^1, x^2, \ldots, x^k)= \alpha (x^1)^{n_1} (x^2)^{n_2} \cdots (x^k)^{n_k}, \,\, \alpha \neq 0, n_i \in \Z_{\geq 0},
\end{align}
and scalings $a_1, \ldots, a_k \in \cals_\infty$, we define the scaling
\begin{align}
\label{def:scalings3}
m(a_1, a_2, \ldots, a_k):= n_1 a_1+ n_2 a_2 + \cdots + n_k a_k.  
\end{align} 
If $\alpha =0$ so that $m\equiv 0$ is the zero polynomial, then we define 
\begin{align*}
m(a_1, a_2, \ldots, a_k) =\infty 
\end{align*}
regardless of the choice of  $a_1, \ldots, a_k \in \cals_\infty$.  We recall that if $p: \RR^k\rightarrow \RR$ is a polynomial, then we can write 
\begin{align}
\label{eqn:sf}
p=\sum_{j=1}^\ell m_j
\end{align}
 for some integer $\ell \geq 1$ and monomials $m_j:\RR^k\rightarrow \RR$, $j=1,2,\ldots, \ell$.  We say that $p$ is written in \emph{standard form} if $\ell$ in~\eqref{eqn:sf} is minimal; that is, $p$ is expressed as a sum of a minimal number of monomials.   This way, no cancellations of monomials can occur.  For any polynomial $p:\RR^k\rightarrow \RR$ written in standard form as in~\eqref{eqn:sf} and any $a_1, \ldots, a_k\in \mathscr{S}_\infty$, we define
 \begin{align*}
 p(a_1, \ldots, a_k) = \min_{j=1,2, \ldots, \ell} m_j(a_1, \ldots, a_k).
 \end{align*}      
 Also, if $p:\RR^k\rightarrow \RR^j$ is a vector of polynomials with $p=(p_1, \ldots, p_j)$, $p_\ell:\RR^k \rightarrow \RR$, then for $a_1, \ldots, a_k \in \mathscr{S}_\infty$ we define 
 \begin{align*}
 p(a_1, \ldots, a_k) = \min_{\ell=1,2,\ldots, j} p_\ell(a_1, \ldots, a_k).  
 \end{align*}

Next, let us elucidate the definitions above. 

\begin{example}
Consider $p: \RR^2 \rightarrow \RR$ given by 
\begin{align}
\label{eqn:fepo}
p(x^1, x^2) = x^1 x^2 +(x^1)^2 -x^1 x^2,  
\end{align}
and suppose that $a_1=(2,0)$, $a_2=(0,1/2)$.  Then $p(a_1, a_2) \neq (2,1/2)$ since $p$ is not in standard form.  The standard form of $p$ is $p(x^1, x^2)= (x^1)^2$ so that $p(a_1, a_2) = (4,0)$.   
\end{example}

\begin{example}Suppose $p:\RR^2\rightarrow \RR^2$ is given by
\label{ex:1}
\begin{align}
\label{eqn:pex}
p(x^1,x^2) =((x^1)^2-(x^2)^2, 2x^1x^2+(x^2)^2+x^1+(x^1)^6). 
\end{align}
Let $a_1=a_2=(1,0)\in \cals$.  Then $p(a_1, a_2) =(1,0).$  Note that, in this case, $p(a_1, a_2)$ coincides with the minimal degree among all nontrivial monomials in each component of $p$.  Alternatively, if $ b_1=(1/2, 1/2)$ and $b_2=(0,1)$, then $p(b_1, b_2)= (0,2)$.  In this case, the terms in $p$ leading to $p(b_1, b_2)=(0,2)$ are the two $(x^2)^2$ terms in~\eqref{eqn:pex}. 
\end{example}

For a fixed set of scalings, we also define what it means for a polynomial vector field to be homogeneous under that collection of scalings.

\begin{definition}
Suppose that $p:\RR^k \rightarrow \RR^\ell$, $p=(p_1, p_2, \ldots, p_\ell)$, is a vector of polynomials and $(a_1, a_2, \ldots, a_k) \in \cals_\infty^k$ is such that $p(a_1, \ldots, a_k) \neq \infty$.  We say that $p$ is \emph{homogeneous under} $(a_1, \ldots, a_k)$ \emph{of degree} $p(a_1, \ldots, a_k)$ if for every $j=1,2,\ldots, \ell$ and every monomial $m_j$ appearing in the standard form for $p_j$, we have $m_j(a_1, \ldots, a_k) = p(a_1, \ldots, a_k)$. 
\end{definition}

\begin{example}
Suppose $p(x^1,x^2)=((x^1)^2-(x^2)^2, 2x^1x^2)$ and $a_1= a_2=(1,0)$.  Then $p$ is homogeneous under $(a_1, a_2)$ of degree $(2,0)$.  Note that this coincides with our usual understanding of homogeneity.  However, if $b_1=(0,1)$ and $b_2=(1/2,0)$, then $p$ is not homogeneous under $(b_1, b_2)$ since all monomials in its standard form have distinct scalings under $(b_1, b_2)$.    
\end{example}

Next, we show how the notion of scaling relates to the actual scaling under a given small parameter.  That is, for any $\alpha\in (0, e^{-1})$ and any $a =(a^1, a^2) \in \cals$, define 
\begin{align}
\label{def:scalings4}
\alpha^a:= \alpha^{a^1} (\log \log \alpha^{-1})^{a^2}. 
\end{align}

\begin{lemma}
\label{lem:hom}
Fix $k, \ell \in I$ and suppose $p: \RR^k \rightarrow \RR^\ell$ is a vector of polynomials that is homogeneous under $(a_1, \ldots, a_k) \in \cals_\infty^k$ of degree $p_*:= p(a_1, \ldots, a_k)\neq \infty$.  Then for any $\epsilon \in (0, e^{-1})$ and $x=(x^1, \ldots, x^k) \in \RR^k$ we have 
\begin{align*}
p\big( \epsilon^{a_1} x^1 , \ldots,  \epsilon^{a_k} x^k \big) = \epsilon^{p_*} p(x^1, \ldots, x^k).
\end{align*} 
\end{lemma}

\begin{proof}
It suffices to prove the result for monomials $m:\RR^k\rightarrow \RR$ of the form 
\begin{align}
m(x^1, \ldots, x^k)= \alpha (x^1)^{n_1} (x^2)^{n_2} \cdots (x^k)^{n_k}, \,\, \alpha \neq 0, n_i \in \Z_{\geq 0}.
\end{align}
Clearly, any such monomial is homogeneous under any $(a_1, \ldots, a_k)\in \cals_\infty^k$ of degree $m_*= m(a_1, \ldots, a_k)=n_1 a_1+\cdots + n_k a_k$ provided $m(a_1, \ldots, a_k)\prec \infty$.   By the condition $m(a_1, \ldots, a_k)\prec \infty$, we may assume without loss of generality that $(a_1, \ldots, a_k) \in \cals^k$.  Finally, observe that for any $\epsilon \in (0, e^{-1})$ we have 
\begin{align*}
m \big( \epsilon^{a_1} x^1 , \ldots,  \epsilon^{a_k} x^k \big)=\epsilon^{n_1 a_1+ \cdots + n_k a_k} m(x^1, \ldots, x^k) = \epsilon^{m_*} m(x^1, \ldots, x^k). 
\end{align*}  
\end{proof}

\section{Laws of the iterated logarithm at time zero}
\label{sec:IL}
In this section, we use the pair scalings from $\mathscr{S}_\infty$ introduced in~\eqref{def:scalings2} to determine a functional law of the iterated logarithm for the process~\eqref{eqn:SDEp} at time zero.  We proceed by induction, starting from the asymptotically dominant directions, and then propagating to the remaining directions with the help of the ordering $\preceq$ introduced in Section~\ref{sec:scales}.   Below, we extend the definition of $\proj_1$ to include in its domain the point $\infty\in \mathscr{S}_\infty$ so that $\proj_1 \infty :=\infty$.  Also, we use the usual generalization of $(<, >, \leq , \geq )$ to the extended real numbers $\RR \cup \{ \infty\}$, so that we can compare $\proj_1 a$ and $\proj_1 b$ whenever $a, b \in \mathscr{S}_\infty$.  
 
\subsection{Determination of scalings}
\label{sec:ILdet}
Recall $P, \sigma$ as in~\eqref{eqn:SDEp} and the set $I=\{1,2,\ldots, n\}$. To initialize the inductive procedure, first define 
\begin{align}
I_0= \{ j \in I \,: \, \sigma^j >0\}.
\end{align}
Note that $I_0$ contains the indices of the directions in which noise is acting independently in~\eqref{eqn:SDEp}.
For every $j \in I_0$, define $a_{0,j} \in \cals$ and $P_\text{L}^j:\RR^n\rightarrow \RR$  by
\begin{align}
a_{0,j}=(1/2, 1/2) \qquad \text{ and } \qquad P^j_\text{L}(x^1, \ldots, x^n)=0.  
\end{align}
Observe that for $j\in I_0$, $P^j_{\text{L}}$  is clearly independent of all variables $x^1, x^2, \ldots, x^n$.  For $j\notin I_0$, we set $a_{0,j}= \infty$.  Also note that the scaling $(1/2, 1/2)$ corresponds to the correct scale for the Brownian motion at time zero, in the law of the iterated logarithm sense.    


Inductively, for $\ell \geq 0$ if $I_\ell\neq \emptyset$ and if either $I\setminus \cup_{k=0}^\ell I_k= \emptyset$ or  $P^j(a_{\ell,1}, \ldots, a_{\ell,n})= \infty$ for all $j\in I \setminus\cup_{k=0}^\ell I_k$, the procedure stops and we set $I_{j}=\emptyset$ for $j\geq \ell+1$.  Otherwise, define
\begin{align*}
m_{\ell+1} &= \min \{ \proj_1 P^j (a_{\ell,1}, \ldots, a_{\ell,n}) \,: \, j \in I \setminus\cup_{k=0}^\ell I_k\},\\
I_{\ell+1}&= \{ j \in I \setminus \cup_{k=0}^\ell I_k \, : \,  \proj_1 P^j(a_{\ell,1}, \ldots, a_{\ell,n}) =m_{\ell+1}\},
\end{align*} 
and set
\begin{align}
\label{eq:iadf}
a_{\ell+1, j} = \begin{cases}
P^j(a_{\ell,1}, \ldots, a_{\ell,n}) +(1,0)& \text{ if } j \in I_{\ell+1}\\
 a_{\ell,j} & \text{otherwise}
\end{cases}.
\end{align}
For $j\in I_{\ell+1}$, let $P^j_{\text{L}}(x^1, \ldots, x^n)$ be the homogeneous polynomial in standard form of degree $P^j(a_{\ell,1}, \ldots, a_{\ell,n})$ which satisfies
\begin{align}
\label{eq:elopo}
P^j_\text{L}(a_{\ell,1}, \ldots, a_{\ell,n}) \prec (P^j-P^j_\text{L})(a_{\ell,1}, \ldots, a_{\ell,n}).  
\end{align}
Note that, by definition, for $j\in I_{\ell+1}$, $P^j_\text{L}(x)= P^j_\text{L}(\pi_{\cup_{k=1}^\ell I_k} (x))$; that is, $P^j_\text{L}$, $j\in I_{\ell+1}$, is a function depending only on the variables whose indices belong to $\cup_{k=1}^\ell  I_k$.

\begin{definition}
\label{def:dimQsig}
If there exists $\ell\geq 0$ such that 
\begin{align}
\bigcup_{k=0}^\ell I_k=I
\end{align}
 we call the system~\eqref{eqn:SDEp} \emph{noise propagating}.  Otherwise, we call the system~\eqref{eqn:SDEp} \emph{noise defective}.  If~\eqref{eqn:SDEp} is noise propagating, let $\text{dim}(P,\sigma)$ be the smallest nonnegative integer such that 
 \begin{align*}
 \bigcup_{k=0}^{\text{dim}(P,\sigma)} I_k =I
 \end{align*} 
 and define, for $j\in I$,  
\begin{align}
\label{eqn:fraka}
\mathfrak{a}_j:= a_{\text{dim}(P,\sigma), j}.  
\end{align}
 We call $\text{dim}(P,\sigma)$ the \emph{dimension of propagation}.  When $P$ and $\sigma$ are clear from context, we simply write $\dim$ for $\dim(P, \sigma)$.    
\end{definition}

\begin{example}
\label{ex:LD1}
Suppose that $n=2k$ for some $k\in \N$ and  $U:\RR^k\rightarrow \RR$ is a polynomial.  Consider the equation on $\RR^n$ for $x_t=(q_t, p_t) \in \RR^k \times \RR^k$ given by 
\begin{align}
\label{eqn:LD}
\begin{cases}
dq_t = p_t \, dt & \\
dp_t = -p_t \, dt -\nabla U(q_t) \, dt + \sqrt{2} \, dB_t & \\
x_0=(q_0, p_0)=(0,0)
\end{cases}
\end{align}  
where $B_t=(B_t^1, \ldots, B_t^k)$ is a standard $k$-dimensional Brownian motion.  Here, $I_0= \{ k+1,\ldots, 2k\}$ and $I_1= \{ 1,2, \ldots, k \}$.  Thus the system~\eqref{eqn:LD} is noise propagating with $\dim=1$.  Furthermore,
\begin{align}
P^j_\text{L}(q,p)=
\begin{cases}
0 & \text{ if } j \in I_0 \\
p^j & \text{ if } j \in I_1
\end{cases}
\qquad \text{ and } \qquad 
\mathfrak{a}_j=
\begin{cases}
(1/2, 1/2) & \text{ if } j \in I_0 \\
(3/2,1/2) & \text{ if } j \in I_1.
\end{cases}
\end{align} 
\end{example}
\begin{example}
\label{ex:LD2}
Consider the same setup as in Example~\ref{ex:LD1}, but this time assume that the process $x_t=(q_t, p_t)$ starts from a general initial condition $x_0 =x\in \RR^n$ with $p_0 \neq 0$; that is, $x_t$ satisfies
\begin{align}
\begin{cases}
dq_t = p_t \, dt \\
dp_t= -p_t \, dt - \nabla U(q_t) \, dt + \sqrt{2} \, dB_t \\
x_0 =(q_0, p_0)\in \RR^n.
\end{cases}
\end{align}
Let $\bar{q}_t= q_t -q_0$ and $\bar{p}_t = p_t -p_0$ and note that $\bar{x}_t:=(\bar{q}_t, \bar{p}_t)$ solves
\begin{align}
\label{eqn:LDshift}
\begin{cases}
d\bar{q}_t = (\bar{p}_t+p_0)\, dt \\
d\bar{p}_t = -(\bar{p}_t +p_0) \, dt - \nabla U(\bar{q}_t+q_0) \, dt + \sqrt{2} \, dB_t \\
(\bar{q}_0, \bar{p}_0)=(0,0) \in \RR^n.
\end{cases}
\end{align}
Thus, for $\bar{x}_t$ the system~\eqref{eqn:LDshift} is also noise propagating with $I_0=\{ k+1, \ldots, 2k \}$, $I_1= \{ j \in I \setminus I_0 \, : \, p_0^j \neq 0 \}$ and $I_2 = \{ j \in I \setminus (I_0\cup I_1) \, : \, p_0^j =0 \}$. Furthermore, $\dim=2$ if there exists $j$ such that $p_0^j\neq 0$.  Otherwise, $\dim =1$.  Also, $P^j_\text{L}(x)=0$ and $\mathfrak{a}_j=(1/2,1/2)$ for $j\in I_0$.  However, this time for $j\notin I_0$ we have \begin{align*}
P^j_\text{L}(q,p) =\begin{cases}
p_0^j & \text{ if } p_0^j \neq 0 \\
p^j & \text{ if } p_0^j =0
\end{cases}\qquad \text{ and } \qquad \mathfrak{a}_j = \begin{cases}
(1,0) & \text{ if } p_0^j \neq 0 \\
(3/2,1/2) & \text{ if } p_0^j =0 .
\end{cases} 
\end{align*}  
\end{example}

\begin{remark}
\label{rem:LD}
In Example~\ref{ex:LD1} and Example~\ref{ex:LD2}, it is not necessary to assume that $U:\RR^n\rightarrow \RR$ is a polynomial. Indeed, so long as $U\in C^2(\RR^n)$ the same conclusions hold since $B_t$ still dominates the $p$-directions at time zero.    
\end{remark}

\begin{example}
\label{ex:lor96}
In this example, we consider a stochastically perturbed Lorenz '96 model.  The model describes the evolution of $x_t=(x_t^1, \ldots, x_t^n)$, $n\geq 4$, satisfying 
\begin{align}
\label{eqn:stochlor1}
\begin{cases}
dx^i_t = [(x^{i+1}_t-x^{i-2}_t) x^{i-1}_t - x^i_t] + \sigma^i \,dB^i \\
x^i_0=0
\end{cases}
\end{align}
where $\sigma^1, \sigma^2 >0$ and $\sigma^i =0$ for $i=3,\ldots, n$. In~\eqref{eqn:stochlor1}, $B_t=(B^1_t, \ldots, B^n_t)$ is a standard $n$-dimensional Brownian motion, and we assume that the coordinates satisfy the periodcity conditions
\begin{align*}
x^{-1}:= x^{n-1},\qquad x^{0}:=x^n \qquad\text{ and } \qquad x^{n+1}=x^1. 
\end{align*}
In this case, $I_0=\{ 1,2\}$ and $I_{j-2}= \{ j\}$ for $j=3,4,\ldots, n$ with $P_\text{L}^{1}=P_\text{L}^2=0$ and 
\begin{align}
\label{eqn:PLlor}
P_\text{L}^j (x)= \begin{cases}
- x^{j-2}x^{j-1} & \text{ if } j=3,4,\ldots, n -1\\
x^1 x^{n-1} & \text{ if } j=n.
\end{cases}
\end{align}
Thus the system~\eqref{eqn:stochlor1} is noise propagating with $\dim=n-2$.  Also,  $\mathfrak{a}_1=\mathfrak{a}_2=(1/2,1/2)$ and $\mathfrak{a}_j$, $j=3,4,\ldots, n-1$, satisfy the Fibonacci-like pattern
\begin{align*}
\mathfrak{a}_j = \mathfrak{a}_{j-1}+ \mathfrak{a}_{j-2}+ (1,0).
\end{align*}  
The scaling $\mathfrak{a}_n$, however, satisfies a slightly different relation 
\begin{align*}
\mathfrak{a}_n = (3/2, 1/2)+ \mathfrak{a}_{n-1}.
\end{align*}
\end{example}

\begin{remark}
Observe that if we assumed that $\sigma^1=0$ or $\sigma^2=0$ and $\sigma_3=\cdots = \sigma_n=0$ in~\eqref{eqn:stochlor1}, then the system is noise defective. 
\end{remark}

\begin{example}
\label{ex:sabra}
Next, we consider finite-dimensional approximations of the Sabra model, which was originally proposed as a shell model for turbulence~\cite{sabra_98}.  Fixing real parameters $\delta \in (0,2) \setminus \{ 1\}$ and $\alpha^i , \beta^i \geq 0$, the finite-dimensional projection of the shell-model is given, according to~\cite{BL_22}, by the following evolution of $u=(u^1, u^2, \ldots, u^J) \in \mathbf{C}^J$:
\begin{align}
\label{eqn:sabra}
\begin{cases}
d u^m= i 2^{m} \Big( \overline{u_{m+1}} u_{m+2} -\delta \overline{u_{m-1}} u_{m+1} - \frac{\delta-1}{4} u_{m-2} u_{m-1} \Big) \, dt - \delta 2^{2m} u^m \, dt + \alpha^m \, dB^{m}+i \beta^m \, dW^m\\
u_0 =0 \in \mathbf{C}^J,
\end{cases} 
\end{align} 
where $\{ B^m, W^m\, : \, m=1,2,\ldots, J\}$ is a collection of mutually independent standard, real-valued Brownian motions.  We assume that $u^{-1}=u^0=u^{J+1}=u^{J+2}=0$ and that $\alpha^1, \beta^1, \alpha^2, \beta^2 >0$, while $\alpha^m=\beta^m=0$ for $2\leq m \leq J$.  Furthermore, although each component $u^m$ in the projected Sabra model is complex-valued and, strictly speaking, the calculations performed above concern only real values, we will adapt the scaling analysis above and maintain the complex-valued nature of the equation.     

Under our assumptions, it follows that~\eqref{eqn:sabra} is noise propagating with complex-valued vector field $P_\text{L}$ defined by
\begin{align}
P_\text{L}^m(u) =\begin{cases}
0 & \text{ if } m=1,2 \\
-i 2^m \frac{\delta-1}{4} u^{m-2} u^{m-1} & \text{ if } m=3,\ldots, J
\end{cases}
\end{align}
Furthermore, because the real and imaginary parts of every direction in this equation have the same scaling, we can extend the definition $\mathfrak{a}_m$ to be the scaling for the complex direction $u^m$.  In particular, 
\begin{align*}
\mathfrak{a}_m= \begin{cases}
(1/2,1/2) & \text{ if } m=1,2 \\
\mathfrak{a}_{m-2}+ \mathfrak{a}_{m-1}+ (1,0) & \text{ if } m=3,\ldots, J.
\end{cases}
\end{align*}

\end{example}

\begin{example}
\label{ex:NPNH}
This example shows that noise propagation can be weaker than hypoellipticity as in~\cite{Hor_67}.  Consider the following SDE on $\RR^3$ started from the origin:
\begin{align}
\label{eqn:NP}
\begin{cases}
dx^1&= dB^1\\
dx^2&=x^1 \, dt\\
dx^3&= x^1 \, dt.
\end{cases} 
\end{align}
Then~\eqref{eqn:NP} is noise propagating with $\dim =1$, but the process lives on the subset $$\{ (x,y,y) \,: \, x,y\in \RR \}\subset \RR^3.$$  Thus the associated Markov kernel does is not absolutely continuous with respect to Lebesgue measure on $\RR^3$.  Hence, the diffusion is not hypoelliptic. 
\end{example}

\subsection{Consequences of scaling structure}
Suppose that the system~\eqref{eqn:SDEp} is noise propagating.  For $\epsilon \in (0, e^{-1})$, define a family of processes $x_{\epsilon,t}= (x_{\epsilon,t}^1, \ldots, x_{\epsilon,t}^n)$ by
\begin{align}
\label{eqn:LILrescaling}
x_{\epsilon,t}^j = \frac{x_{\epsilon t}^j}{\epsilon^{\mathfrak{a}_j}}.
\end{align} 
where $\mathfrak{a}_j$ is as in~\eqref{eqn:fraka}.  Also, we let $P_\text{L}\in \Poly(\RR^n)$ be such that its $j$th component is $P^j_\text{L}$.  Define\begin{align}
\label{eqn:BMscaled}
B_{\epsilon,t}= \frac{B_{\epsilon t}}{\sqrt{\epsilon}}. 
\end{align}

One of our main results in this section is the following theorem, which states that the scaling calculated in the previous section gives the correct form at the level of the stochastic differential equation to have a functional law of the iterated logarithm.  This connection is further elaborated below.      
\begin{theorem}
\label{thm:np}
Suppose that the system~\eqref{eqn:SDEp} is noise propagating.  Then for all $t< \tau_\infty/\epsilon$, $x_{\epsilon,t}$ solves an SDE of the form
\begin{align}
\label{eqn:LILeps}
d x_{\epsilon, t} = [P_\text{\emph{L}}(x_{\epsilon, t}) + R_{\epsilon, \text{\emph{L}}}(x_{\epsilon, t})] \, dt + \frac{\sigma dB_{\epsilon,t}}{\epsilon^{(0, 1/2)}} 
\end{align}
where $R_{\epsilon, \text{\emph{L}}}\in \Poly(\RR^n)$ is such that for all $C>0$, $\sup_{|x|\leq C} |R_{\epsilon, \text{\emph{L}}}(x) |\rightarrow 0$ as $\epsilon \rightarrow 0^+$.
\end{theorem}

In order to prove Theorem~\ref{thm:np}, we first need a technical auxiliary result.  

\begin{lemma}
\label{lem:auxin}
Suppose~\eqref{eqn:SDEp} is noise propagating.  Then $\mathfrak{a}_m\prec \mathfrak{a}_j$ whenever $m\in I_\ell$ and $j \in I_{\ell'}$ with $\ell < \ell'$.  Furthermore, for any $j\in I\setminus I_0$, we have 
\begin{align*}
P^j_{\text{\emph{L}}}(\mathfrak{a}_1, \ldots, \mathfrak{a}_n) \prec (P^j-P^j_{\text{\emph{L}}})(\mathfrak{a}_1, \ldots, \mathfrak{a}_n). 
\end{align*}  
\end{lemma}

Lemma~\ref{lem:auxin} is a consequence of the following result. \begin{proposition}
\label{prop:compare}
Let $2\leq k\leq n$, $a_1, \ldots, a_{k-1} \in \cals$, and suppose $q_1, q_2:\RR^n \rightarrow \RR$ are polynomials with 
\begin{align}
\label{eqn:qoqt}
q_1(a_1, \ldots, a_{k-1}, \infty_{n-k}) \prec q_2(a_1, \ldots, a_{k-1}, \infty_{n-k})
\end{align}
where $\infty_{n-k} \in (\mathscr{S}_\infty)^{n-k}$ is given by $\infty_{n-k}=(\infty, \infty, \ldots, \infty)$. 
If $a_k, \ldots, a_n \in \cals$ are such that $$q_1(a_1, \ldots, a_{k-1}, \infty_{n-k}) \prec a_m$$ for every $m=k, \ldots, n$, then  
\begin{align*}
q_1(a_1, \ldots, a_n)= q_1(a_1, \ldots, a_{k-1}, \infty_{n-k}) \prec q_2(a_1, \ldots, a_n). 
\end{align*}
\end{proposition}
\begin{proof}
Letting $0_{n-k}$ denote the zero vector in $\RR^{n-k}$, observe that we can write
\begin{align*}
q_1(x^1, \ldots, x^n) = R_{11}(x^1, \ldots, x^{k-1}, 0_{n-k})+ \sum_{j=k}^n x^j R_{1j}(x^1, \ldots, x^n),\\
q_2(x^1, \ldots, x^n)= R_{21}(x^1, \ldots, x^{k-1}, 0_{n-k})+ \sum_{j=k}^n x^j R_{2j}(x^1, \ldots, x^n)
\end{align*}
for some polynomials $R_{ij}$.  If we set $\hat{R}_{11}(x_1, \ldots, x_{k-1})= R_{11}(x_1, \ldots, x_{k-1}, 0_{n-k})$ and \\$\hat{R}_{21}(x_1, \ldots, x_{k-1}) = R_{21}(x_1, \ldots, x_{k-1}, 0_{n-k})$, then $q_1(a_1, \ldots, a_{k-1}, \infty_{n-k}) \prec a_m$ for every $m=k,\ldots, n$ implies  
\begin{align*}
q_{1}(a_1, \ldots, a_{k-1}, \infty_{n-k}) = \hat{R}_{11}(a_1, \ldots, a_{k-1}) = q_1(a_1, \ldots, a_{n}).  
\end{align*}  
On the other hand, either
\begin{align*}
q_2(a_1, \ldots, a_n) = \hat{R}_{21}(a_1, \ldots, a_{k-1}) = q_2(a_1, \ldots, a_{k-1}, \infty_{n-k}), 
\end{align*}
in which case the conclusion follows from~\eqref{eqn:qoqt},
or  for some $m\in \{k,\ldots, n \}$
\begin{align*}
a_m \preceq q_2(a_1, \ldots, a_n) =a_m+ \hat{R}_{2m}(a_1, \ldots, a_n). 
\end{align*}
and the conclusion follows from $q_1(a_1, \ldots, a_{k-1}, \infty_{n-k})\prec a_m$.  
\end{proof}

\begin{proof}[Proof of Lemma~\ref{lem:auxin}]
For the first statement, by transitivity of $\prec$ it suffices to show that $\mathfrak{a}_m \prec \mathfrak{a}_j$ if $m \in I_{\ell}$ and $j \in I_{\ell + 1}$.  If $\ell = 0$,  then $\mathfrak{a}_m = (1/2, 1/2)$ and 
by \eqref{eq:iadf} we have
$\mathfrak{a}_j \succeq (1, 0)$,  and therefore $\mathfrak{a}_m \prec \mathfrak{a}_j$.  If $\ell \geq 1$,  then without loss of generality, assume that 
\begin{align*}
\bigcup_{s = 0}^{\ell - 1} I_{s} = \{1, \ldots,  k\} \qquad  \text{ and } \qquad  
I_\ell = \{k+1, \ldots,  k + j_\ell\}
\end{align*}
 for some $k \geq 1$ and $j_\ell \geq 0$.     
By \eqref{eq:iadf},  we obtain 
$\mathfrak{a}_j  = a_{\ell+1,j} = P^j(a_{\ell, 1}, \ldots, a_{\ell, n}) + (1, 0)$.  Then,  there are polynomials $Q : \RR^k \to \RR$,  $(R_m)_{m = k+1}^{k+j_\ell} : \RR^n \to \RR$,  $(S_m)_{m = k + j_\ell + 1}^n  : \RR^n \to \RR$ such that 
\begin{equation}\label{eq:ggeae}
P^j(x^1, \ldots,  x^n) = Q(x^1, \ldots, x^k) + \sum_{i = k+1}^{k+j_\ell} x^i R_i(x^1,  \ldots,  x^n) +   \sum_{i = k+j_\ell+1}^{n} x^i S_i(x^1,  \ldots,  x^n) \,. 
\end{equation}
Since $i \not \in \bigcup_{s = 0}^{\ell} I_{s}$ for $i > k+ j_\ell$, we obtain that $\proj_1 (a_{\ell, i} S_i(a_{\ell, 1}, \ldots, a_{\ell, n})) = \infty$, whenever $S_i \neq 0$, and therefore
we do not have to consider the last sum in  \eqref{eq:ggeae}. 
If $Q \neq 0$,  then
\begin{align*}
 \proj_1 P^j(a_{\ell, 1}, \ldots, a_{\ell, n}) &= \min_{i \in \{k+1, \ldots, k+j_\ell\}, R_i \neq 0} \{  \proj_1 Q(a_{\ell, 1}, \ldots, a_{\ell, k}),   \proj_1( a_{\ell, i} +  R_i(a_{\ell, 1}, \ldots, a_{\ell, n})) \} \\
&\geq
\min_{i \in \{k+1, \ldots, k+j_\ell\},  R_i \neq 0} \{ \proj_1 Q(a_{\ell - 1, 1}, \ldots, a_{\ell - 1, k}),    \proj_1 a_{\ell, i} \} \\
&>
\min_{i \in \{k+1, \ldots, k+j_\ell\},  R_i \neq 0} \{ \proj_1a_{\ell,m} - 1,   \proj_1 a_{\ell, i} \} = \proj_1a_{\ell, m} - 1
\,,
\end{align*}
where we used that $a_{\ell, r} = a_{\ell - 1,  r}$ for $r \in \{1, \ldots,  k\} = \bigcup_{s = 0}^{\ell - 1} I_{s}$ and the fact that $Q$ is not a minimizer in the definition of $m_{\ell}$.    
If $Q = 0$,  then there is an $i$ such that $R_i \neq 0$,  and 
we obtain 
\begin{align*}
 \proj_1 P^j(a_{\ell, 1}, \ldots, a_{\ell, n}) &= \min_{i \in \{1, \ldots, k\}, R_i \neq 0} \{ \proj_1( a_{\ell, i} +  R_i(a_{\ell, 1}, \ldots, a_{\ell, n})) \} \\
&>
\min_{i \in \{1, \ldots, k\}, R_i \neq 0} \{   \proj_1 a_{\ell, i} \} >  \proj_1 a_{\ell, m} 
\,,
\end{align*}
and the first statement follows from $ \proj_1 a_{\ell+1,j} =\proj_1 ( P^j(a_{\ell, 1}, \ldots, a_{\ell, n})) + 1  > \proj_1 a_{\ell, m}$.  

For the second statement fix any $j \in I \setminus I_0$.    Since the system is noise propagating, there exists $\ell \geq 0$ such that $j \in I_{\ell+1}$, and there exists $P^j_L$ such that \eqref{eq:elopo} holds.  
Assume without loss of generality that $\bigcup_{i = 0}^\ell I_i = \{1, \ldots,  k\}$ for some $k \geq 1$, for otherwise we permute the indices.  
By \eqref{eqn:fraka} we have that 
$a_{\ell, m} = \mathfrak{a}_m$ for any $m \leq k$ and by the definition of $I_\ell$, it follows that $\mathfrak{a}_{\ell, m} = \infty$ if $m >k$. 

Define $q_1 = P_\text{L}^j$ and $q_2 = P^j - P^j_\text{L}$ and by \eqref{eq:elopo},  the assumption \eqref{eqn:qoqt} holds and by the definition of $P_{\text{L}}^j$
\begin{align*}
q_1 (\mathfrak{a}_1, \ldots, \mathfrak{a}_{k}, \infty_{n - k}) =  P^j_\text{L}(a_{\ell, 1}, \ldots, \mathfrak{a}_{\ell, n}) = P^j(a_{\ell, 1}, \ldots, a_{\ell, n})  = a_{\ell+1, j} - (1, 0)\,.
\end{align*}
In particular,  for any $j \in I_{\ell + 1}$ we have $q_1 (\mathfrak{a}_1, \ldots, \mathfrak{a}_{k}, \infty_{n - k}) \prec a_{\ell+1, j} = \mathfrak{a}_j$.  For $j \in I_{\ell'}$ with $\ell' \geq \ell + 1$, 
$q_1 (\mathfrak{a}_1, \ldots, \mathfrak{a}_{k}, \infty_{n - k}) \prec \mathfrak{a}_j$  follows from the first statement of the lemma.  Finally, the second assertion of the lemma  
is a consequence of Proposition~ \ref{prop:compare}.   
\end{proof}

\begin{proof}[Proof of Theorem~\ref{thm:np}]
Supposing that $\epsilon t < \tau_\infty$, we separate the proof into two cases.  

Case 1: $j\in I_0$.  Then we have
\begin{align*}
d x_{\epsilon,t}^j  &= \frac{\epsilon}{\epsilon^{(1/2, 1/2)}} P^j(x_{\epsilon t}) \, dt + \frac{1}{\epsilon^{(0,1/2)}} \sigma^j dB_{\epsilon, t}^j \\
&= \frac{\epsilon}{\epsilon^{(1/2,1/2)}} P^j\big( \epsilon^{\mathfrak{a}_1} x_{\epsilon,t}^1, \ldots, \epsilon^{\mathfrak{a}_n}  x_{\epsilon, t}^n\big)\, dt + \frac{1}{\epsilon^{(0,1/2)}} \sigma^j dB_{\epsilon, t}^j .\end{align*}  
Recalling that $P^j_\text{L}(x)\equiv 0$ when $j\in I_0$, we define 
\begin{align*}
R_{\epsilon, \text{L}}^j(x) =\frac{\epsilon}{\epsilon^{(1/2,1/2)}} P^j\big( \epsilon^{\mathfrak{a}_1}x^1, \ldots, \epsilon^{\mathfrak{a}_n} x^n \big)\end{align*} 
and note that, since $\epsilon/\epsilon^{(1/2, 1/2)}\rightarrow 0$ as $\epsilon \rightarrow 0^+$, $R_{\epsilon, \text{L}}^j\rightarrow 0$ as $\epsilon \rightarrow 0^+$ uniformly on compact subsets of $\RR^n$. 

Case 2: $j\notin I_0$.  In this case, we note that 
\begin{align*}
dx_t^j = P^j(x_t) \, dt
\end{align*}  
since $\sigma^j=0$ for all $j\notin I_0$.  Hence, 
\begin{align*}
d x_{\epsilon,t}^j &= \frac{\epsilon}{\epsilon^{\mathfrak{a}_j}} [P^j_\text{L}(x_{\epsilon t}) + (P^j-P^j_\text{L})(x_{\epsilon t})]\, dt.
\end{align*} 
Next, by construction and Lemma~\ref{lem:hom}, we have 
\begin{align*}
\frac{\epsilon}{\epsilon^{\mathfrak{a}_j}} P^j_{\text{L}} \big(\epsilon^{\mathfrak{a}_1} x^1, \ldots, \epsilon^{\mathfrak{a}_n} x^n) =  P_\text{L}^j(x^1, \ldots, x^n),  
\end{align*}
so that, by homogeneity, 
\begin{align*}
\frac{\epsilon}{\epsilon^{\mathfrak{a}_j}} P^j_\text{L}(x_{\epsilon t})= P^j_\text{L}( x_{\epsilon, t}^1, \ldots, x_{\epsilon,t}^n)= P^j_\text{L}(x_{\epsilon,t}). 
\end{align*}
Next, define for any $x\in \RR^n$ and $\epsilon >0$
\begin{align*}
R_{\epsilon, \text{L}}^j(x):= \frac{\epsilon}{\epsilon^{\mathfrak{a}_j}} (P^j-P^j_\text{L})( \epsilon^{\mathfrak{a}_1}x^1, \ldots, \epsilon^{\mathfrak{a}_n} x^n).
\end{align*}
Since  
$$P^j_\text{L}(\mathfrak{a}_1, \ldots, \mathfrak{a}_n) \prec (P^j-P^j_\text{L})( \mathfrak{a}_1, \ldots, \mathfrak{a}_n)$$ by Lemma~\ref{lem:auxin},
it follows that $R_{\epsilon, \text{L}}^j(x)\rightarrow 0$ as $\epsilon \rightarrow 0^+$ uniformly on compact subsets of $\RR^n$.  
\end{proof}

Considering the structure of equation~\eqref{eqn:LILeps}, at first glance it may appear as though we have not captured the dynamics correctly as $\epsilon \rightarrow 0^+$ due to the presence of the small term, namely $1/\epsilon^{(0,1/2)}$, in front of the Brownian motion.  However, this is not true.  In fact, we are able to characterize the set of limit points of the trajectories $t\mapsto x_{\epsilon, t}$ as $\epsilon \rightarrow 0^+$.

More precisely, suppose that~\eqref{eqn:SDEp} is noise propagating and consider the following control problem
\begin{align}
\label{eqn:controlLIL}
\begin{cases}
\dot{x}=  P_\text{L}(x) + \sigma \dot{f}\\
x_0=x\in \RR^n
\end{cases}
\end{align}
where $f\in\HH$.  That is, we consider the family of trajectories $t\mapsto \varphi_t(P_\text{L}, f)x$ for $f\in \HH$.  Because $P_\text{L}$ is a polynomial vector field, and thus may contain nonlinearities, it is not completely clear that $t\mapsto \varphi_t(P_\text{L}, f)x$ is defined for all times $t\in[0,1]$.  However, $P_\text{L}$ has a specific structure and we claim that~\eqref{eqn:controlLIL} is explicitly integrable, which in particular implies that $t\mapsto \varphi_{ t}(P_\text{L},f)x$ is defined on $[0,1]$ for any $f\in \HH$.  Specifically, fixing $f\in \mathscr{H}$, recall that for $j\in I_0$ we have $P_\text{L}^j=0$.  Hence, for $j\in I_0$
\begin{align}
\label{eqn:iterate3cont}
\varphi_{t}^j(P_\text{L},f)x = x^j+ \sigma^j f_t^j,   
\end{align}
which is well-defined and bounded for all times $t\in [0,1]$.  Next, for $j\in I_{k+1}$ with $1\leq k+1\leq \dim(P, \sigma)$, we have $\sigma^j=0$ and 
\begin{align}
\label{eqn:iterate1cont}
P^j_\text{L}(x)=P^j_\text{L}\big(\pi_{\cup_{\ell=0}^{k}I_\ell} (x)\big)
\end{align}
so that
\begin{align}
\label{eqn:iterate2cont}
\varphi_{t}^j(P_\text{L},f)x =x^j+ \int_0^t P^j_\text{L}(\pi_{\cup_{\ell=0}^{k}I_\ell} (\varphi_{s}(P_\text{L}, f)x)) \, ds.
\end{align}
Proceeding by induction, we obtain that the integrand is a bounded function, and therefore $\varphi_t^j(P_L, f)x$ is bounded for any $j$.  
Thus, $t\mapsto \varphi_{t}(P_\text{L}, f)x$ is well-defined on $[0,1]$ for any $f\in \HH, \, x\in \RR^n$ by the structure of $P_\text{L}$.  

Define the \emph{Cramer transform} $\lambda :\HH\rightarrow [0, \infty]$  and a subset $\CMcal{I}\subset \HH$ by 
\begin{align*}
\lambda (g)= \inf\bigg\{\tfrac{1}{2} \textstyle{\int_0^1 |\dot{f}_s|^2 \, dt} \, : \, f\in \HH \text{ and } \varphi_{\cdot}(P_\text{L}, f)0=g_{\cdot} \bigg\} \quad \text{ and } \quad
\CMcal{I}= \{ g \in \HH \, : \, \lambda(g) \leq 1 \}. 
\end{align*}
As a corollary of Theorem~\ref{thm:np}, we obtain the following result which relates the solution of~\eqref{eqn:LILeps} and the trajectories of the control problem~\eqref{eqn:controlLIL}.
\begin{theorem}
\label{thm:LIL}
Suppose that the equation~\eqref{eqn:SDEp} is noise propagating.  Then there exists a constant $\epsilon_*\in (0, e^{-1})$ such that for $\PP$-almost every $\omega$:
\begin{itemize}
\item[(i)] The set $\mathcal{X}(\omega) := \{ x_{\epsilon, \cdot}(\omega) \}_{\epsilon \in (0, \epsilon_*)}$ is relatively compact in $\E$.
\item[(ii)] $d(x_{\epsilon, \cdot}(\omega), \mathscr{I}) \rightarrow 0$ as $\epsilon \rightarrow 0^+$.
\item[(iii)]  For every $h\in \CMcal{I}$, there exists a subsequence $\epsilon_j(\omega, h) \downarrow 0$ as $j\rightarrow \infty$ such that 
\begin{align*}
d( x_{\epsilon_j(\omega,h)}(\omega), h) \rightarrow 0\quad \text{ as }\quad  j\rightarrow \infty. 
\end{align*}
\end{itemize}
\end{theorem}

\begin{proof}
This result follows by combining Theorem~\ref{thm:np} above with~\cite[Theorem 2.6]{CFH_22}.  The fact that the noise in equation~\eqref{eqn:LILeps} is additive allows us to check Assumption 2 in the statement of~\cite[Theorem 2.6]{CFH_22} by way of \cite[Corollary 2.10]{CFH_22}.  
\end{proof}

\begin{remark}
Observe that Theorem~\ref{thm:LIL} shows that the set of limit points of the path $t\mapsto x_{\epsilon,t}(\omega)$ as $\epsilon \rightarrow 0^+$ is for almost all $\omega$ the set of trajectories $\mathscr{I}$. \end{remark}

For $t\in [0,1]$, define the collection of points
\begin{align}
\label{eqn:accessbounded}
\A_{\text{L},t}(x)  =\{y \in \RR^n \, : \, \varphi_{t}(P_\text{L}, f)x=y \text{ for some } f\in \HH_{1} \}.
\end{align}
By evaluation at time $t=1$ in the mapping $t\mapsto x_{\epsilon, t}$, it follows by Theorem~\ref{thm:LIL} that the set of limit points as $\epsilon \rightarrow 0^+$ of the sequence of random variables $x_{\epsilon, 1}$ is given by 
\begin{align}
\label{eqn:limitptfixed}
\cl(\A_{\text{L},1}(0)).
\end{align}
Finding practical methods to characterize points in the set $\A_{\text{L},1}(0)$ is of major importance, which is discussed below in Section~\ref{sec:control}.

\section{Scaling in law}
\label{sec:lawscale}
The goal of this section is to establish the asymptotic behavior of the process at time zero in the sense of distributions.  Similar to Section~\ref{sec:ILdet}, we use scalings $\mathscr{S}_\infty$ as introduced in~\eqref{def:scalings2}, but we neglect logarithmic corrections, which corresponds to the second component in the scaling.  In particular, we employ \emph{power scalings}; that is, $a =(a^1, a^2)\in \cals$ with $a^2=0$.  As in Section~\ref{sec:ILdet}, we proceed inductively in order to identify the correct scale in each direction.  Although the results are, not surprising, similar, we remark that the power scaling algorithm can retain more terms compared to the case of the LIL scaling, and therefore the limiting dynamics may be different.  See Example~\ref{ex:rDr} for a concrete example.

\subsection{Determination of scalings}
First recall that $I_0\subset I$ is given by  
\begin{align*}
I_0= \{ j \,: \,\sigma^j > 0\}.
\end{align*}
For any $j \in I$, define $b_{0,j} \in \cals_\infty$ by
\begin{align}
b_{0,j} = \begin{cases}
(1/2,0) & \text{ if } j \in J_0\\
\infty & \text{ otherwise}.  
\end{cases}
\end{align}
For any $j\in I_0$, define $P_{\text{D}}^j(x^1, \ldots, x^n)=0$.  Also, note that $\proj_1 b_{0, j}= \proj_1 a_{0, j}$ for all $j$.


Inductively, for $\ell \geq 0$ if either $I\setminus \cup_{k=0}^\ell I_k= \emptyset$ or $P^j(b_{\ell, 1}, \ldots, b_{\ell, n}) =\infty$ for all $j\in I\setminus \cup_{k=0}^\ell I_\ell$, then the procedure stops and we set $I_j =\emptyset$ for $j\geq \ell+1$.  Otherwise, recalling the definitions of $m_{\ell+1}$ and $I_{\ell+1}$ we have 
\begin{align*}
m_{\ell+1}&=\min \{ \proj_1 P^j(b_{\ell,1}, \ldots, b_{\ell,n}) \, : \, j \in I \setminus \cup_{k=0}^\ell I_k \},\\
I_{\ell+1}&=\{ j \in  I \setminus \cup_{k=0}^\ell I_k \, : \, \proj_1 P^j(b_{\ell,1}, \ldots, b_{\ell,n}) =m_{\ell+1}\}. \end{align*}
Set
\begin{align}
b_{\ell+1,j} = \begin{cases}
(m_{\ell+1},0) +(1,0) & \text{ if }  j \in I_{k+1}\\
b_{\ell,j} & \text{ otherwise}.
\end{cases}
\end{align}
For $j\in I_{\ell+1}$, let $P^j_{\text{D}}$ denote the polynomial in $P^{j}(x^1, \ldots, x^n)$ which is homogeneous of degree $(m_{\ell+1},0)$ and such that $$ (m_{\ell+1},0) \prec (P^j - P^j_{\text{D}})(b_{\ell,1}, \ldots, b_{\ell,n}) .$$

Supposing that the system is noise propagating, the inductive process above must stop and we recall that 
\begin{align*}
 \bigcup_{k=0}^{\dim(P, \sigma)} I_k=I\end{align*}
where $\dim(P, \sigma)$ is the dimension of propagation.  For $j\in I$, setting
\begin{align}
\mathfrak{b}_j = b_{\dim(P,\sigma), j},
\end{align}
we see that $\mathfrak{b}_j \in \cals$ and $\proj_1 \mathfrak{b}_j = \proj_1 \mathfrak{a}_j$ for all $j$ by construction. 

\begin{remark}
It can be readily shown that for $j\in I_0 \cup I_1$, $P^j_\text{L} = P^j_{\text{D}}$.  Furthermore, for all $j$, $P^j_{\text{D}}$ contains all terms in $P^j_\text{L}$.  The second assertion follows since it is a weaker requirement for terms to have the same scaling under power scalings.    
\end{remark}

\begin{example}
\label{ex:rDr}This example shows that, in general, $P_L \neq P_D$.  
Consider the following SDE on $\RR^4$
\begin{align*}
dx^1&= dB^1\\
dx^2 &= dB^2\\
dx^3&= x^1 x^2 \, dt\\
dx^4&=x^1 (x^3)^2+ x^3 (x^1)^5 \, dt 
\end{align*}
where $B^1_t, B^2_t$ are independent, standard Brownian motions on $\RR$ and the solution starts at the origin in $\RR^4$.  Then $I_0= \{1,2\}$, $I_1=\{ 3\}$, and $I_2= \{4\}$ so that $\dim=2$.  Also, $P^1_\text{L}=P^2_\text{L}=P^1_\text{D}=P^2_\text{D}=0$, $P^3_\text{L}=P_\text{D}^3=x^1x^2$, while $P^4_\text{L}= x^3(x^1)^5$ and $P^4_\text{D}=x^1 (x^3)^2+ x^3(x^1)^5$.

\end{example}

\begin{example}
\label{ex:lor92part2}
We saw that the stochastically perturbed Lorenz '96 model in Example~\ref{ex:lor96} is noise propagating with $\dim=n-2$.  Moreover, we have that $P_\text{D}=P_\text{L}$ with $\mathfrak{b}_j =(\mathfrak{a}_j^1, 0)$ 
for any $j=1,2,\ldots, n$, where the $\mathfrak{a}_j$'s are as in Example~\ref{ex:lor96}. 
\end{example}

\begin{example}
\label{ex:sabra2}
Returning to the context of the projected Sabra model in Example~\ref{ex:sabra}, we already saw that the system is noise propagating.  Moreover, in this case we also have that $P_\text{D}=P_\text{L}$ with $\mathfrak{b}_m =(\mathfrak{a}_m^1, 0)$ 
for any $m=1,2,\ldots, J$, where the $\mathfrak{a}_j$'s are as in Example~\ref{ex:sabra}. 
\end{example}
\subsection{Consequences of the distributional rescaling}
Supposing that the system~\eqref{eqn:SDEp} is noise propagating, introduce a vector field $P_\text{D}\in \Poly(\RR^n)$ with $j$th component given by $P_\text{D}^j$ and process $t\mapsto y_{\epsilon,t}\in \E$ defined by
\begin{align}
\label{eqn:rescaledist}
y_{\epsilon, t} = \big( \tfrac{x_{\epsilon t}^1}{\epsilon^{\mathfrak{b}_1}}, \ldots,\tfrac{x_{\epsilon t}^n}{\epsilon^{\mathfrak{b}_n}}\big).
\end{align}
Our next goal is to show that $t\mapsto y_{\epsilon,t}$ converges as $\epsilon \rightarrow 0^+$ in distribution on the space $\E$.  To prove this claim, we first state a lemma, which follows from Proposition~\ref{prop:compare} in a similar way to Lemma~\ref{lem:auxin}. 
\begin{lemma}
\label{lem:compare2}
Suppose that the system~\eqref{eqn:SDEp} is noise propagating.  Then for any $j\in I$
\begin{align}
P^j_\text{\emph{D}}(\mathfrak{b}_1, \ldots, \mathfrak{b}_n) \prec (P^j-P^j_\text{\emph{D}})( \mathfrak{b}_1, \ldots, \mathfrak{b}_n). 
\end{align}
\end{lemma} 
Recalling that $B_{\epsilon,t}= B_{\epsilon t}/\sqrt{\epsilon}$, we have the following result which gives a natural candidate for the distributional limit of $t\mapsto y_{\epsilon, t}$ in $\E$ as $\epsilon\rightarrow 0^+$.  We formalize this assertion in Corollary~\ref{cor:dist} below. 
\begin{theorem}
\label{thm:dist}
Suppose that equation~\eqref{eqn:SDEp} is noise propagating and recall the random variable $\tau_\infty$ is the time of existence for the solution of~\eqref{eqn:SDEp}.  Then for any $t<\tau_\infty/\epsilon$ we have
\begin{align}
\label{eqn:Deps}
\begin{cases}
dy_{\epsilon, t} = P_\text{\emph{D}}(y_{\epsilon,t})\, dt  + R_{\epsilon, \text{\emph{D}}}(y_{\epsilon,t}) \, dt + \sigma dB_{\epsilon,t}\\
y_{\epsilon,0}=0
\end{cases}
\end{align}
where $B_{\epsilon,t}$ is as in~\eqref{eqn:BMscaled}  and $R_{\epsilon, \text{\emph{D}}}\in \Poly(\RR^n)$ is such that $\sup_{|x| \leq C} |R_{\epsilon, \text{\emph{D}}}(x)|\rightarrow 0$ as $\epsilon \rightarrow 0^+$ for all $C>0$.  
\end{theorem}
\begin{proof}
 
Fix $t<\tau_\infty/\epsilon$. As in the proof of Theorem~\ref{thm:np}, we split the proof into two cases.

Case 1:  If $j\in I_0$, then since $\mathfrak{b}_j =(1/2, 0)$, we have 
\begin{align*}
dy_{\epsilon,t}^j &= \epsilon^{1/2} P^j(x_{\epsilon t}) \, dt + \sigma^j \, dB_{\epsilon, t}^j \\
&= \epsilon^{1/2} P^j(\epsilon^{\mathfrak{b}_1} y_{\epsilon,t}^1, \ldots, \epsilon^{\mathfrak{b}_n} y_{\epsilon,t}^n) \, dt + \sigma^j \, dB_{\epsilon, t}^j . 
\end{align*} 
Setting $R_{\epsilon, \text{D}}^j (y):= \epsilon^{1/2} P^j(\epsilon^{\mathfrak{b}_1}y^1, \ldots, \epsilon^{\mathfrak{b}_n} y^n)$ we find that $R_{\epsilon, \text{D}}^j \rightarrow 0$ uniformly on compact subsets of $\RR^n$.  This finishes the proof in Case 1. 

Case 2:  Suppose that $j \notin I_0$ so that $\sigma^j=0$. By construction,
\begin{align*}
dy_{\epsilon,t}^j = \frac{\epsilon}{\epsilon^{\mathfrak{b}_j}} P^j(x_{\epsilon t}) \, dt &=  \frac{\epsilon}{\epsilon^{\mathfrak{b}_j}} P^j_{\text{D}}(x_{\epsilon t}) \, dt +  \frac{\epsilon}{\epsilon^{\mathfrak{b}_j}}(P^j- P^j_{\text{D}})(x_{\epsilon t}) \, dt\\
&= P^j_\text{D} (y_{\epsilon,t})\, dt +  \frac{\epsilon}{\epsilon^{\mathfrak{b}_j}}(P^j- P^j_{\text{D}})(x_{\epsilon t}) \, dt\\
&=:P^j_\text{D} (y_{\epsilon,t})\, dt + R_{\epsilon, \text{D}}^j(y_{\epsilon, t})\, dt \end{align*} 
By construction and Lemma~\ref{lem:compare2}, we find that $R_{\epsilon, \text{D}}^j\rightarrow 0$ as $\epsilon \rightarrow 0^+$ uniformly on compact sets.  This finishes the proof.  

 \end{proof}

 \begin{corollary}
\label{cor:dist}
Suppose that \eqref{eqn:SDEp} is noise propagating.  Then the process $t\mapsto y_{\epsilon,t}\in \E$, $t\in[0,1]$, solving~\eqref{eqn:Deps} converges in law to the solution of the SDE
\begin{align}
\label{eqn:SDEdist}
\begin{cases}
dy_t= P_\text{\emph{D}}(y_t) \, dt + \sigma\, dB_t\\
y_0=0.
\end{cases}
\end{align}  
Furthermore, the solution to~\eqref{eqn:SDEdist} is nonexplosive.

\end{corollary}  
\begin{proof}
The second conclusion follows from the structure of the vector field $P_\text{D}$ following a similar reasoning used in \eqref{eqn:controlLIL}-\eqref{eqn:iterate3cont}.  In particular, the system~\eqref{eqn:SDEdist} is explicitly integrable.  

For the first conclusion, consider the following system
\begin{align}
\label{eqn:zprocess}
\begin{cases}
dz_{\epsilon,t} = P_\text{D}( z_{\epsilon,t}) + R_{\epsilon, \text{D}} (z_{\epsilon,t}) + \sigma \, dB_t \\
z_{\epsilon,0}= 0.
\end{cases}
\end{align}
Note that since $B_t$ and $B_{\epsilon,t}$ have the same distribution, $z_{\epsilon,t}$ has the same distribution in the space $\E$ as $ y_{\epsilon,t}$.  However, $ z_{\epsilon, t}$ is driven by the same Brownian motion as $y_t$.   
Let $T_{\epsilon,k}= \inf\{ t \geq 0 \, : \, |z_{\epsilon,t}|\geq k \}$ and $T_k= \inf\{ t\geq 0 \, : \, |y_t| \geq k \}$.  Then for any $t\leq T_{\epsilon,2k} \wedge T_k \wedge 1$ we have that
\begin{align*}
\sup_{s\in [0,t]} | z_{\epsilon,s}- y_s| &\leq  t \sup_{|z| \leq 2k} |R_{\epsilon, \text{D}}(z)| +\int_0^t \sup_{v\in [0,s]} |P_\text{D}(z_{\epsilon,v}) - P_\text{D}(y_v)| \, ds \\
&\leq t \sup_{|z| \leq 2k} |R_{\epsilon, \text{D}}(z)|  + C_k  \int_0^t \sup_{v\in [0,s]} |z_{\epsilon,v}- y_v| \, ds\end{align*}     
for some constant $C_k >0$.  Gr\"{o}nwall's inequality then implies that for any $t\leq T_{\epsilon,2k} \wedge T_k \wedge 1$
\begin{align}
\label{eqn:inequalityzprocess}
\sup_{s\in [0,t]} |z_{\epsilon, s}- y_s| \leq  t e^{C_k t}  \sup_{|z| \leq 2k} |R_{\epsilon, \text{D}}(z)| 
\end{align} 
Since by Theorem~\ref{thm:dist} we have $R_{\epsilon, \text{D}}\rightarrow 0$ as $\epsilon\rightarrow 0^+$ uniformly on compact subsets of $\RR^n$, taking $\epsilon \rightarrow 0^+$ we find that $\liminf_{\epsilon\rightarrow 0^+} T_{\epsilon,2k}\geq T_k$.  Consequently, $d_t(z_{\epsilon, \cdot}, y_{\cdot})\rightarrow 0$ as $\epsilon \rightarrow 0^+$ provided $t \leq T_k \wedge 1$.  Since $y_t$ is nonexplosive, for almost every $\omega \in \Omega$ choose $k=k(\omega)$ such that $T_k(\omega)>1$, and the result is proven.   

\end{proof}

Analogous to the LIL scaling and motivated by Theorem~\ref{thm:LIL}, we also consider a control problem related to the distributional rescaling~\eqref{eqn:Deps}
\begin{align}
\label{eqn:controllaw}
\begin{cases}\dot{x}= P_\text{D}(x) + \sigma \dot{f}\\
x_0=x
\end{cases}
\end{align} 
where the control $f\in \HH$.  However, this time, the control problem is not restricted to controls with a bounded norm in $\HH$ (see relation~\eqref{eqn:accessbounded}).  This makes it noticeably easier to solve the control problem.  

By the same reasoning used above in~\eqref{eqn:iterate1cont}-\eqref{eqn:iterate2cont}, the structure of the vector field $P_\text{D}$ ensures that for all initial conditions $x\in \RR^n$,  the solution $\varphi_t(P_\text{D}, f) x$ of equation~\eqref{eqn:controllaw} is well-defined on $[0,1]$.  For $t>0$, we define the set of accessible points by 
\begin{align}
\label{def:AD}
 \A_{\text{D},t}(x)&= \{ y \in \RR^n \, : \, \varphi_{t}(P_\text{D}, f)x=y \text{ for some } f\in \HH\}. 
\end{align}
The relationship between~\eqref{eqn:Deps} and~\eqref{eqn:controllaw}  is described in the next result (see also~\cite{SV_72ii, SV_72i}).

\begin{lemma}
\label{lem:support}
Suppose that the system~\eqref{eqn:SDEp} is noise propagating.  Then the support of the process $t\mapsto y_t$ in $\C$ coincides with 
\begin{align}
\text{\emph{cl}}(\{t\mapsto \varphi_{t}(P_{\text{\emph{D}}}, f)0\, : \, f\in \HH \})
\end{align}
where the closure is taken in the space $\C$.      
\end{lemma}

\begin{proof}
Note that if $f\in \C$, then we can also define $\varphi_{t}(P_\text{D},f)0$ as the solution of the integral equation
\begin{align}
\varphi_{t}(P_\text{D},f)0= \int_0^t P_\text{D}(\varphi_{s}(P_\text{D},f)0) \, ds + \sigma f_t .
\end{align}
An inductive argument similar to the one given in \eqref{eqn:iterate1cont}-\eqref{eqn:iterate2cont} implies that the mapping $f\mapsto \varphi_{\cdot}(P_\text{D},f)0:\C\rightarrow \C$ is continuous.  Since the support of $t\mapsto B_t$ in $\C$ coincides with the closure of $\HH$ in the topology $\C$, the result follows by continuity.    
\end{proof}

For any $\epsilon >0$, let $S_{\epsilon, \text{D}}:\RR^n \rightarrow \RR^n$ be given by
\begin{align}
S_{\epsilon, \text{D}} (x) = \bigg( \frac{x^1}{\epsilon^{\mathfrak{b}_1}},  \frac{x^2}{\epsilon^{\mathfrak{b}_2}}, \ldots,  \frac{x^n}{\epsilon^{\mathfrak{b}_n}}\bigg).
\end{align}
We also have the following implication for the densities provided they exist and have the required regularity.    
\begin{corollary}
Suppose that equation~\eqref{eqn:SDEp} is noise propagating and $x_t$ is nonexplosive with probability density function $(t,\eta)\mapsto q_t(\eta)\in C((0, \infty) \times \RR^n ; [0, \infty))$.  If the solution $y_t$ of~\eqref{eqn:SDEdist} has probability density function $(t,\eta) \mapsto q_{t, \text{\emph{D}}}(\eta): (0 ,\infty) \times \RR^n \rightarrow [0, \infty)$, then for any $t>0$ and $y\in \RR^n$
\begin{align*}
\epsilon^{\mathfrak{b}_1+\cdots + \mathfrak{b}_n} q_{\epsilon t}( S_{\epsilon, \text{\emph{D}}}^{-1}(y)) \rightarrow q_{t, \text{\emph{D}}}(y) \qquad \text{ as } \qquad \epsilon \rightarrow 0^+.
\end{align*}   
\end{corollary}

\begin{proof}
Fix $\alpha^j< \beta^j$, $j=1,2,\ldots n$, and consider the family of rectangles
\begin{align*}
A_\epsilon = [\alpha^1 \epsilon^{\mathfrak{b}_1}, \beta^1 \epsilon^{\mathfrak{b}_1}]\times \cdots \times [\alpha^n \epsilon^{\mathfrak{b}_n}, \beta^n \epsilon^{\mathfrak{b}_n}], \qquad \epsilon >0.\end{align*}
Note that for any $t>0$ we have by Corollary~\ref{cor:dist}
\begin{align*}
\PP \{ x_{\epsilon t} \in A_\epsilon\}= \PP \{ S_{\epsilon, \text{D}}(x_{\epsilon t}) \in A_1\} = \PP \{ z_{\epsilon,t} \in A_1\} \rightarrow \PP \{ y_t \in A_1\}= \int_{A_1} q_{t, \text{D}}(\eta) \, d\eta 
\end{align*}
as $\epsilon \rightarrow 0^+$. 
On the other hand, 
\begin{align*}
\PP \{ x_{\epsilon t} \in A_\epsilon\} = \int_{A_\epsilon} q_{\epsilon t}(\eta) \, d\eta = \int_{A_1} \epsilon^{\mathfrak{b}_1+ \cdots + \mathfrak{b}_n} q_{\epsilon t} ( S_{\epsilon, \text{D}}^{-1}(\eta) ) \, d\eta
\end{align*}
Hence as $\epsilon \rightarrow 0^+$
\begin{align*}
\int_{A_1} \epsilon^{\mathfrak{b}_1+ \cdots + \mathfrak{b}_n} q_{\epsilon t} ( S_{\epsilon, \text{D}}^{-1}(\eta) ) \, d\eta \rightarrow  \int_{A_1} q_{t, \text{D}}(y) \, dy
\end{align*}
Since the choice of $\alpha^j< \beta^j$, $j=1,2,\ldots n$, in the definition of $A_\epsilon$ was arbitrary, continuity of the densities implies the claimed pointwise convergence.  \end{proof}

%
%

\section{H\"{o}rmander's condition, the Gramian and relationships between the control problems}
\label{sec:control}
Throughout this section, we suppose that the system~\eqref{eqn:SDEp} is noise propagating.  Our goal is to explore the  control problems~\eqref{eqn:controlLIL} and~\eqref{eqn:controllaw} associated to each scaling limit.  Specifically, we will determine when the control sets $A_{\text{L}, t}(x)$ and $A_{\text{D},t}(x)$ contain an open set and relate the two, when possible.  Methods that can be used to find points in these sets are deferred to the following section.

To unify the discussion below between the different control sets and different vector fields, fix $\Q\in\{ P_\text{L}, P_\text{D}\}$ and consider the control problem
\begin{align}
\label{eqn:controlgeneral}
\begin{cases}
\dot{x}=\Q (x)+\sigma \dot{f}\\
x_0=x\in \RR^n
 \end{cases}
\end{align}
where $x\in \RR^n$ and $f\in \HH$ with $\HH$ defined in~\eqref{eqn:Hdef}.  For $0<t\leq 1$, we let 
\begin{align}
\label{eqn:smallcont}
\A^1_t(x)&= \{ y\in \RR^n \, : \, \varphi_t(\Q,f)x=y \text{ for some } f\in \HH_{1}  \},\\
\label{eqn:largecont}
\A^2_t(x)&= \{ y\in \RR^n \, : \, \varphi_t(\Q,f)x=y \text{ for some } f\in \HH\}.
\end{align}
Clearly, $\A^1_t(x) \subset \A^2_t(x)$ for any $0<t\leq 1$ and $x\in \RR^n$.

\subsection{H\"{o}rmander's Condition and the Gramian}   We first determine when $\A_t^1(x)$ contains a non-empty open set in $\RR^n$.  To this end, we use the \emph{Gramian} matrix associated to the control problem~\eqref{eqn:controlgeneral}, as follows.  

Fix $x\in \RR^n $, $0<t\leq 1$ and $f, f_1,\ldots,  f_n \in \HH$.  Define a matrix $F:=\begin{bmatrix}  f_1& \cdots & f_n\end{bmatrix}$ where $f_i$, $i=1,2,\ldots, n$, is viewed as a column vector.  Define the mapping $\alpha\mapsto \tilde{g}_{t,F}(\alpha,f, x):\RR^n\rightarrow \RR^n$ by 
\begin{align*}
\tilde{g}_{t, F}(\alpha,f,  x) = \varphi_t(Q, f + F \alpha)x
\end{align*}
where $F\alpha$ is the usual product of the matrix $F$ and the column vector $\alpha\in \RR^n$. 
Letting $\tilde{g}_{t,F}'(f,x):= D_\alpha \tilde{g}_{t,F}(\alpha,f,x) |_{\alpha =0}$, we observe that $\tilde{g}'_{t,F}(f,x)$ satisfies the $n\times n$ matrix-valued ODE
\begin{align}
\label{eqn:megpf}
\tilde{g}'_{t,F} (f,x)= \int_0^t D Q(\varphi_s(\Q, f) x) \tilde{g}'_{s,F}(f,x) \, ds + \sigma F_t .
\end{align}
In order to obtain a representation for the solution of~\eqref{eqn:megpf} and isolate its dependence on $F$, consider the $n\times n$ matrix-valued solutions $J_t(f,x)$ and $K_t(f,x)$ of the linear equations
\begin{align}
J_t (f,x)&= I + \int_0^t DQ(\varphi_s(\Q,f)x) J_s(f,x) \, ds,\\
K_t(f,x)&= I - \int_0^t K_s(f,x) DQ(\varphi_s(\Q,f)x) \, ds.
\end{align}
One can check that 
\begin{align*}
\tfrac{d}{dt}(K_t (f,x) J_t(f,x)) =0\qquad \text{ and } \qquad K_0(f,x) J_0(f,x)=I 
\end{align*}
so that both $K_t(f,x)$, $J_t(f,x)$ are invertible for all $t\in [0,1]$ and $K_t(f,x)=J_t(f,x)^{-1}$.  Next, by uniqueness of solutions of ODEs, the matrix $\tilde{g}'_{t,F}(f,x)$ can be expressed in terms of the matrices $K_t(f,x)$ and $J_t(f,x)$ as \begin{align*}
\tilde{g}'_{t,F}(f,x) = J_t(f,x) \int_0^t K_s(f,x) \sigma \dot{F}_s \, ds 
\end{align*}
Note that, up to this point, $f_1, f_2, \ldots, f_n\in \HH$ were arbitrary.  We now fix them so that our matrix $F=F^\mu(f)$ is given by  
\begin{align}
\label{eqn:Fchoice}
F_t^\mu(f)= \frac{1}{\mu}\int_0^t \sigma^* K_s(f,x)^*  \, ds  
\end{align} 
where $*$ denotes the Hermitian transpose and $\mu>0$ is determined below.  Since $t\mapsto K_t(f,x)$ solves a linear equation with bounded coefficients, it is also bounded, and therefore the choice~\eqref{eqn:Fchoice} ensures that $f_1, \ldots, f_n \in \HH$.    
Define
\begin{align*}
g'_{\mu,t} (f,x) = \tilde{g}'_{t, F^\mu(f)}(f,x). 
\end{align*}
and note that 
\begin{align*}
g'_{\mu,t}(f,x) = \frac{1}{\mu} J_t(f,x)\int_0^t K_s(f,x) \sigma \sigma^* K_s(f,x)^* \, ds. 
\end{align*}
We are interested in invertibility of $g'_{\mu, t}(f,x)$, which is equivalent to invertibility of 
\begin{align*}
G_{\mu,t}(f,x) :=  \frac{1}{\mu}\int_0^t K_s(f,x) \sigma \sigma^* K_s^*(f,x) \, ds
\end{align*}
for any $\mu >0$. If $\mu=1$, we call $G_t(f,x):= G_{1,t}(f,x)$ the \emph{Gramian} associated to the control problem~\eqref{eqn:controlgeneral} with control $f\in \HH$.  Observe that $G_t(f,x)$ is $n\times n$ symmetric, nonnegative-definite matrix. 
\begin{theorem}
\label{thm:gramm}
Suppose that there exists $0<\lambda <1$, $s>0$, and $f\in \HH_{\lambda}$ such that the matrix $G_s(f,x)$ is invertible.  Then $\A_t^1(x)$ contains a nonempty, open set in $\RR^n$ for all $t\geq s$. 
\end{theorem}  

\begin{proof}
Let $\lambda,s >0$, $f\in \HH_\lambda$ be as in the statement of the result.  Note that $G_s(f,x)$ being invertible is equivalent to the matrix 
$G_{\mu, s}(f,x)$ being invertible for all $\mu >0$.  Since $f\in \HH_\lambda$ with $\lambda <1$ and $f_1, f_2, \ldots, f_n \in \HH$ are fixed in~\eqref{eqn:Fchoice}, there is $\mu=\mu(\lambda) >0$ large enough such that 
\begin{align*}
\frac{1}{2} \int_0^1 |\dot{f}_s|^2+ \sum_{i=1}^n |\dot{f}_{i,s}|^2 \, ds  \leq 1.   
\end{align*} 
Next note that for $ t\geq s$, $G_{\mu, t}(f,x)$ is invertible since for $\mu >0$ and $y\in \RR^n_{\neq 0}$
\begin{align*}
G_{\mu, t}(f,x) y \cdot y = \frac{1}{\mu} \int_0^t | \sigma^* K_v(f,x)^* y |^2 \, dv &\geq \frac{1}{\mu} \int_0^s| \sigma^* K_v(f,x)^* y |^2\, dv\\
&= G_{\mu,s}(f,x) y\cdot y >0,
\end{align*}
where in the last step we used symmetry and positive definiteness of $G_{\mu, s}(f,x)$.  Finally, since $\alpha \mapsto \varphi_t(Q, f+ F^\mu(f) \alpha)x$ has derivative $G_{\mu,t}(f,x)$ at $\alpha =0$, the assertion follows from the inverse function theorem. 

\end{proof}

\begin{remark}
The statement of Theorem~\ref{thm:gramm} does not specify whether the non-empty, open set in $\A_t^1(x)$ contains the initial point $x$ of the control problem~\eqref{eqn:controlgeneral}.  In general, $x$ need not be contained in $A_t^1(x)$ for any $t>0$ due to preferred directions in the system~\eqref{eqn:controlgeneral}. Below in Theorem~\ref{thm:transfer}, we provide a condition that guarantees $\A_t^1(0)$ contains an open ball about the origin. \end{remark}

To obtain a condition ensuring the invertibility of the Gramian $G_s(f,x)$ for some fixed $f\in \HH_\lambda$, $s>0$, and $0<\lambda<1$, we consider the following SDE on $ \RR^{2n}$
\begin{align}
\label{eqn:SDEbig}
\begin{cases}
dx_t = dW_t\\
dy_t= \Q(y_t) \, dt + \lambda\sigma  x_t \, dt \\
(x_0, y_0)=(0,x) \in \RR^{2n} 
\end{cases}
\end{align} 
where $W_t$ is a standard, $n$-dimensional Brownian motion.  
Note that the projection of this system onto the process $y_t$ corresponds to the choice of  
\begin{align}
\label{eq:scof}
f_t(\omega) = \int_0^t \lambda W_s (\omega) \, ds
\end{align}
 in~\eqref{eqn:controlgeneral}. Note that $t \mapsto f_t(\omega)$ belongs to $\HH$ for almost all $\omega$.  In fact, the controls of the form~\eqref{eq:scof} are dense in $\HH$, and hence we can view~\eqref{eqn:SDEbig} as being equivalent to the general control problem~\eqref{eqn:controlgeneral}.  
 
For fixed $\lambda \in (0,1)$, we introduce vector fields $Y_{0, \lambda} \in \text{Poly}(\RR^{2n})$ and $Y_\ell \in \text{Cons}(\RR^{2n})$, $\ell=1,2,\ldots,n$ (see Section~\ref{sec:NotTerm} for the definition of $\text{Poly}(\RR^{2n})$ and $\text{Cons}(\RR^{2n})$), associated to~\eqref{eqn:SDEbig} given by 
 \begin{align*}
 Y_{0, \lambda}=(\Q(y)+ \lambda \sigma x)\cdot \nabla_y \qquad \text{ and } \qquad Y_\ell = \partial_{x_\ell}, \,\,\, \ell=1,2,\ldots, n. 
 \end{align*}  
 If $[\,\cdot\,, \,\cdot\,]$ denotes the commutator of vector fields, then standard calculations produce  \begin{align*}
 [Y_\ell, Y_{0, \lambda}]= \lambda \sigma^\ell \partial_{y^\ell}, \,\,\, \ell=1,2,\ldots, n.
 \end{align*}
Define $Z_0, Z_1, \ldots, Z_n\in \Poly(\RR^n)$ by 
  \begin{align*}
 Z_0= \Q(y) \cdot \nabla_y, \qquad \text{ and } \qquad Z_\ell=\sigma^\ell \partial_{y^\ell}, \,\,\ell=1,2,\ldots, n, \end{align*} 
and consider the list of vector fields on $\RR^n$
 \begin{align}
 \label{eqn:Hlist}
 &Z_{\ell_1},  &&\ell_1=1,2,\ldots,n \\
 \nonumber &[Z_{\ell_1}, Z_{\ell_2}],  &&\ell_1,\ell_{2}=0,1,2,\ldots,n \\
 \nonumber &[[Z_{\ell_1}, Z_{\ell_2}], Z_{\ell_3}],  &&\ell_{1}, \ell_2, \ell_3 =0,1,2,\ldots,n\\
 \nonumber &\qquad \vdots  &&\qquad\vdots
 \end{align}   
 Whether or not the above list spans the tangent space at $x\in \RR^n$ is crucial for the invertibility of the Gramian, as discussed below.   

Next, consider the following random matrix
  \begin{align}
 \label{eqn:cov1}
 C_{t, \lambda}(\omega) = \int_0^t K_{s, \lambda}(\omega) \begin{bmatrix} I & 0 \\
 0 &0
 \end{bmatrix} K_{s, \lambda}^*(\omega) \, ds
 \end{align} 
 where the $I$ and $0$ denote, respectively, the $n\times n$ identity and the $n\times n$ zero matrix, and 
the $2n\times 2n$ matrix $K_{s, \lambda}(\omega)$ satisfies the integral equation
 \begin{align}
 \label{eqn:cov2}
 K_{s, \lambda}(\omega) = I - \int_0^s K_{v, \lambda}(\omega) \begin{bmatrix} 0 & 0 \\\lambda \sigma& DQ(y_v(\omega))\end{bmatrix} \, ds .
 \end{align}
 The matrix $C_{t, \lambda}$ is called the \emph{Malliavin covariance matrix} corresponding to the process $(x_t, y_t)$ solving~\eqref{eqn:SDEbig}.  Note that $C_{t, \lambda}$ and $K_{t, \lambda}$ also depend on the initial condition $x\in \RR^n$ as in~\eqref{eqn:SDEbig}; however, below we suppress this dependence for simplicity.   
 
 \begin{theorem}
 \label{thm:Malliavin}
Suppose that equation~\eqref{eqn:SDEp} is noise propagating and let $\lambda>0 $, $x\in \RR^n$ and $t\in (0,1]$ be fixed.  
 \begin{itemize}
\item[(i)]  Define the event $E_{\lambda, t}$ by    
\begin{align}
\label{eqn:Flam}
E_{\lambda, t} = \Big\{ \omega \, : \, \sup_{s\in [0,1]} |W_s(\omega) | \leq \sqrt{2} \Big\}\cap \{ \omega \, : \, C_{t, \lambda}(\omega)^{-1} \text{exists} \}.
\end{align}
If $\omega \in E_{\lambda, t}$ and $f$ is given by $f_\cdot = \lambda \int_0^\cdot W_s(\omega) \, ds$, then $f\in \HH_\lambda$ and $G_t(f, x)$ is invertible. 
\item[(ii)]   Suppose that the vector fields~\eqref{eqn:Hlist} span the tangent space at $x$.  Then 
\begin{align*}
\PP\{ \omega \, : \, C_{t, \lambda}(\omega)^{-1} \text{exists} \}=1, 
\end{align*}  
and consequently $\PP(E_{\lambda,t}) >0$ so that $E_{\lambda, t} \neq \emptyset$.    
 \end{itemize}
 \end{theorem} 

\begin{remark}
The idea for the proof below is inspired by the proof of \cite[Proposition 8.4]{MP_06} where the noise is used as a control.  
\end{remark}

 \begin{proof}[Proof of Theorem~\ref{thm:Malliavin}]
 Fix $t\in (0,1]$, $x\in \RR^n$ and $\lambda >0$.  Then the proof of part (ii) follows from Malliavin's proof of H\"{o}rmander's hypoellipticity theorem~\cite{Bell_12, KS_84, KS_85, KS_87, Norris_06, Nualart_09}.

To establish part (i), let $\omega \in E_{\lambda, t}$ and $f_\cdot= \lambda \int_0^{\cdot} W_s(\omega) \,ds$.  Then $\dot{f}_{s} =   \lambda W_s(\omega) $ and $f \in \HH_\lambda$ follows.  Next, we show that $G_t(f,x)$ is invertible.  Write     
\begin{align*}
 C_{t, \lambda}(\omega) = \begin{bmatrix}
 C_{t, \lambda}^1(\omega) & C_{t, \lambda}^2(\omega)\\
 C_{t, \lambda}^3(\omega) & C_{t, \lambda}^4 (\omega)
 \end{bmatrix}
 \qquad \text{ and } \qquad 
 K_{s, \lambda}(\omega) = \begin{bmatrix}
 K^1_{s, \lambda}(\omega) & K^2_{s, \lambda}(\omega)\\
 K^3_{s, \lambda}(\omega) & K^4_{s, \lambda} (\omega)
 \end{bmatrix}
 \end{align*}
 where $C^i_{t, \lambda}(\omega), K^i_{s, \lambda}(\omega)$ are $n\times n$ matrices.  Combining~\eqref{eqn:cov1} and~\eqref{eqn:cov2} and suppressing the dependence on $\omega$ and $\lambda$, we find that the following relations must be satisfied
 \begin{align*}
 C^4_t&= \int_0^t K_s^3 (K_s^3)^* \, ds, \quad K_s^3= -\int_0^s K_v^4 \lambda \sigma \,dv, \quad K^4_s = I- \int_0^t K^4_s DQ(y_v(\omega)) \, dv .  
 \end{align*}
 In particular,
 \begin{align*}
 C_t^4 = \lambda^2\int_0^t \int_0^s K_v^4 \sigma \, dv \int_0^s  \sigma^* (K_w^4)^* \, dw \, ds,
 \end{align*}
and consequently for any $y \in \RR^n_{\neq 0}$ we have \begin{align*}
 0<C_t^4 y\cdot y =\lambda^2  \int_0^t \bigg|\int_0^s \sigma^* (K_v^4)^* y \, dv\bigg|^2 \, ds.  
 \end{align*}
 Note that, in the above, we used the fact that $C_{t}$ is strictly positive-definite. 
Let $s_1\in (0, t]$ be such that 
 \begin{align*}
 \int_0^{s_1} \sigma^* (K_v^4)^* y \, ds \neq 0.  
 \end{align*}
 In particular, there exists a subset $S_2\subset (0, t]$ of positive Lebesgue measure such that  
 \begin{align*}
 \sigma^* (K_{s_2}^4)^* y \neq 0  
 \end{align*}
 for all $s_2\in S_2$.  This in turn implies that 
  \begin{align*}
 \int_0^t K_s^4 \sigma \sigma^* (K_s^4)^*  y \cdot y\, ds  \geq \int_{S_2}  | \sigma^* (K_{v}^4)^* y|^2 \, ds >0. 
 \end{align*}
 Since $G_t(f,x)= \int_0^t K_s^4 \sigma \sigma^* (K_s^4)^* \, ds $ is the Gramian matrix associated to the control problem~\eqref{eqn:controlgeneral} with control $f$, the proof is finished. 
 \end{proof}

  %
%
%
%
%

\subsection{Relating the control sets}
Recalling the definitions~\eqref{eqn:smallcont} and~\eqref{eqn:largecont} of $\A_t^1(0)$ and  $\A_t^2(0)$, note that for $t\in (0,1]$, $\A_t^1(0)\subset \A_t^2(0)$.  In addition, a combination of Theorem~\ref{thm:gramm} and Theorem~\ref{thm:Malliavin} implies that if the vector fields~\eqref{eqn:Hlist} span the tangent space at $0\in \RR^n$, then $\A_t^1(0)$ contains a nontrivial open set for all $0<t\leq 1$.  However, such an open set might not contain the origin for any $0<t\leq 1$ if the dynamics pushes the control system in preferred directions.  To better understand this open set, we first explore the relationship between $\A_t^1(0)$ and $\A_t^2(0)$ using scalings, when possible.

\begin{definition}
We call the control system~\eqref{eqn:controlgeneral} \emph{component homogeneous} if, for every $j$, there exist a constant $\alpha_j \geq 0$ such that \begin{align}
\label{eqn:chom}
\varphi_t^j(\Q,\epsilon f)0 = \epsilon^{\alpha_j}\varphi^j_t(Q,f)0
\end{align}
for all $\epsilon >0$, $f\in \HH$, $t\in(0,1]$, where $\varphi^j$ is the $j$th component of $\varphi$.   
\end{definition}

\begin{remark}
The relationship~\eqref{eqn:chom} is different than the one afforded by the scaling analysis in previous sections, where we scaled both time and space.  
\end{remark}

Our first result provides us with some directions along which the dynamics of a component homogeneous systems moves for an appropriate control $f$. 
\begin{proposition}
\label{prop:chscaling}
Suppose that \eqref{eqn:SDEp} is noise propagating and~\eqref{eqn:controlgeneral} is component homogeneous.  If for some $t\in (0,1]$ and $v\in \RR^n_{\neq 0}$ we have $v\in\text{\emph{cl}}(\A_t^2 (0)),$ then there exist $f \in \HH_{1}$ and $\epsilon >0$ such that $\varphi_t(\Q,f)0\cdot v >0$.  

\end{proposition}

\begin{proof}
By assumption there exists a sequence $f_\ell \in \HH$ such that $\varphi_t(Q,f_\ell)0 \rightarrow v$ as $\ell \rightarrow \infty$.  Choose $k \geq 1$ large enough so that $$v^j \varphi_t^j(\Q,f_k)0 >0$$ for any $j$ for which $v^j \neq 0$.  Since the system is component homogeneous, for any $\epsilon >0$
\begin{align*}
\varphi_t(\Q,\epsilon f_k)0 \cdot v= \sum_{j: v^j \neq 0}v_j \varphi_t^j(\Q,\epsilon f_k)0  = \sum_{j \, : \, v^j \neq 0} \epsilon^{\alpha_j} v^j \varphi_t^j (\Q,f_k)0  >0. 
\end{align*}  
Picking $\epsilon >0$ small enough so that $\epsilon f_k \in \HH_{1}$ finishes the proof. 

\end{proof}

Next, we provide criteria that ensures $\A_t^1(0)\supset B_r(0)$ for some $r>0$.

\begin{theorem}
\label{thm:transfer}
  Suppose that the system \eqref{eqn:SDEp} is noise propagating and~\eqref{eqn:controlgeneral} is component homogeneous.  Suppose furthermore that the list~\eqref{eqn:Hlist} spans the tangent space at $0 \in \RR^n$ and there exists $r_1,t>0$ such that $0 \in \A_{t}^2(x)$ for all $x\in B_{r_1}(0)$.  Then for any $s>0$ small enough, there exists $r_2 >0$ such that $\A_{t+s}^1(0)\supset  B_{r_2}(0)$.      

\end{theorem}

\begin{proof}
Let $r_1, t>0$ be as in the statement of the result, and let $s>0$ be arbitrary. 
We use the system~\eqref{eqn:SDEbig} to generate our control on $[0, s]$.  Define an event 
\begin{align*}
E_s= \bigcap_{k=1}^\infty E_{\frac{1}{k}, s }\cap \bigg\{ \sup_{0\leq v \leq s }|y_{v, k} | < r_1 \bigg\},
\end{align*}
where $E_{\lambda,t}$ is as in~\eqref{eqn:Flam} and $y_{v, k}$ denotes $y_v$ as in~\eqref{eqn:SDEbig} with $\lambda =1/k$ and $x=0$. Also, set 
\begin{align*}
f_1(\cdot, \omega) = \int_0^{\cdot} W_v(\omega) \, dv,
\end{align*}
which belongs to $\HH$ for almost all $\omega$. 
By Theorem~\ref{thm:Malliavin}
\begin{align}
 \PP\bigg(\bigcap_{k=1}^\infty E_{\frac{1}{k},s} \bigg) = \PP\bigg(\bigg\{ \sup_{r\in [0,1]} |W_r| \leq \sqrt{2}\bigg\} \cap \bigcap_{k=1}^\infty \big\{ C_{s,\frac{1}{k}}^{-1} \text{ exists} \big\}\bigg)= \PP\bigg(\bigg\{ \sup_{r\in [0,1]} |W_r| \leq \sqrt{2}\bigg\} \bigg).
\end{align}
Since $x=0$, there exists $\eta >0$ such that for any $k\geq 1$ we have 
\begin{align}
\sup_{0\leq v \leq 1} |y_{v,k}| < r_1
\end{align}
provided $\omega \in \{ \sup_{r\in [0,1]} |W_r| \leq \eta \}$.  We stress that $\eta$ is independent of $k$ and so we have that 
\begin{align*}
 \bigg\{ \sup_{r\in [0,1]} |W_r | \leq \eta\bigg\}\subset \bigcap_{k=1}^\infty \bigg\{\sup_{0\leq v\leq 1} |y_{v,k}| < r_1 \bigg\}.
\end{align*}   
Hence, $E_s$ for $s\leq 1$ has positive probability.  Fix $\omega \in E_s$ and define
\begin{align*}
\varphi_s(Q, f_1(\cdot, \omega)) 0 =y \in B_{r_1}(0).  
\end{align*}
Also, by Theorem~\ref{thm:Malliavin}, we have that $G_s(\tfrac{1}{k}f, 0)$ is invertible. 
By hypothesis, since $0 \in \A_t^2(y)$, there is $f_2\in \HH$ such that $\varphi_{t}(Q, f_2)y =0$. Define
\begin{align*}
f(v) = \begin{cases}
f_1(v, \omega) & \text{ if } 0 \leq v \leq s \\
f_2(v-s) + f_1(s, \omega) & \text{ if } s<v\leq t+s. 
\end{cases}.
\end{align*} 
Since $f$ is continuous, standard arguments imply that $f\in \HH$.  Moreover, $\varphi_{t+s}(Q, f) 0=0$.  Since the system~\eqref{eqn:controlgeneral} is component homogeneous,  it holds that $\varphi_{t+s}(Q, \tfrac{1}{k} f) 0=0$ for all $k\geq 1$.  Choose $k\geq 1$ large enough so that $\frac{1}{k} f \in \HH_{1/2}$.  Applying Theorem~\ref{thm:Malliavin} and the arguments of Theorem~\ref{thm:gramm}, we have that $G_{t+s}(\tfrac{1}{k}f , 0)$ is invertible.  The result now follows from the inverse function theorem as in the proof of Theorem~\ref{thm:gramm}.
\end{proof}
%

The following lemma provides a simple criterion to verify whether system~\eqref{eqn:controlgeneral} is component homogeneous.

\begin{lemma}
\label{lem:comphom}
Suppose that the system~\eqref{eqn:SDEp} is noise propagating and for every $j=1,2,\ldots, n$, $\Q^j$ is a monomial.  Then the control system~\eqref{eqn:controlgeneral} is component homogeneous.   
\end{lemma}

\begin{remark}
Note that each system considered in Example~\ref{ex:LD1}, Example~\ref{ex:LD2} and Example~\ref{ex:lor96} is such that $P_\text{D}^j=P_\text{L}^j$ is a monomial for all $j$.  Since each of these examples is noise propagating, it follows by Lemma~\ref{lem:comphom} that they are also component homogeneous. 
\end{remark}

\begin{proof}[Proof of Lemma~\ref{lem:comphom}]
If $j \in I_0$, then 
\begin{align*}
\varphi_t^j(\Q,\epsilon f) 0 = \epsilon \sigma^j f_t^j =\epsilon  \varphi_t^j(\Q,f)0.
\end{align*}
Next, suppose inductively that there exists $\alpha_j \geq 0$ such that 
\begin{align*}
\varphi_t^j(\Q,\epsilon f)0 = \epsilon^{\alpha_j}\varphi_t^j(\Q,f)0
\end{align*}
 for any $\epsilon >0$, $j\in \bigcup_{\ell=0}^k I_\ell$ and $f\in \HH$.  Let $j\in I_{k+1}$.  Since $Q^j$ is a monomial and $j\in I_{k+1}$, then  
 \begin{align*}
 Q^j(x)= \beta (x^1)^{j_1} \cdots (x^n)^{j_n}
 \end{align*}
 where $\beta \in \RR_{\neq 0}$ and $j_1,\ldots, j_n \in \Z_{\geq 0}$ satisfy $j_i=0$ if $i\notin \cup_{\ell=0}^k I_\ell$.  Let $\nu_i = \alpha_i$ if $i\in \cup_{\ell=0}^k I_\ell$ and $\nu_i=0$ otherwise.  Then\begin{align*}
\varphi_t^j(\Q,\epsilon f) 0= \int_0^t \Q^j \bigg(\pi_{\bigcup_{\ell=0}^k I_\ell}\varphi_s(\Q,\epsilon f)0 \bigg)\, ds&= \int_0^t  \Q^j \Big(  \epsilon^{\nu_1} \varphi_{s}^1(\Q,f)0, \ldots, \epsilon^{\nu_n} \varphi_s^n(\Q,f)0 \Big) \, ds \\
&=\epsilon^{j_1 \nu_1+\cdots + j_n \nu_n} \int_0^t \Q^j(\varphi_s(\Q,f)0) \, ds\\
&= \epsilon^{j_1 \nu_1+\cdots + j_n \nu_n} \varphi_t^j(\Q,f)0,
\end{align*}
as desired.
\end{proof}

\section{Geometric control theory}
\label{sec:gcontrol}
We next discuss methods from geometric control theory that are useful in identifying points in the sets
\begin{align}
\cl(\A_t^2(x)), \,\,\, 0<t\leq 1, \,\, x\in \RR^n.
\end{align}
We also indicate how to characterize points in the sets 
\begin{align}
\A_t^2(x), \,\,\, 0 < t \leq 1, \,\, x\in \RR^n. 
\end{align}
The central idea is to employ large controls on small time intervals, which allows one to uncover simplified directions in which the solution of the control problem can move.  See, for example,~\cite{GHHM_18, HM_15, Jur_97, JK_85}.  Note that such large scalings are possible in this context since the controls belong to $\HH$ rather than $\HH_{1}$.    If the limiting system is component homogeneous, this information can then be transferred back to the restricted sets $\cl(\A_t^1(0)), 0<t\leq 1$ or even $\A_t^1(0)$, $0<t \leq 1$.  See Proposition~\ref{prop:chscaling} and Theorem~\ref{thm:transfer}.

For convenience in this section, we closely follow the presentation in~\cite{GHHM_18} which uses parametrized local semigroups in place of vector fields.  

\subsection{Parameterized continuous local semigroups and the uniform saturate}
Recalling that $\Delta \notin \RR^n$ is a \emph{death state} and fixing an auxiliary metric space $(M, d_M)$, we begin with a definition.
\begin{definition}
\label{def:cls}
We call a mapping 
\begin{align*}
(t,x,\mu) \mapsto \varphi_t^\mu x : [0, \infty) \times \RR^n \times M \rightarrow \RR^n \cup \{ \Delta \}
\end{align*}
 a \emph{parameterized family of continuous local semigroups on} $\RR^n$ \emph{parameterized by} $M$ if for every $x\in \RR^n$ and $\mu \in M$, there exists a time $t_{x,\mu} \in (0, \infty]$, called the \emph{time of existence}, such that the following conditions are satisfied:
\begin{itemize}
\item[(i)]  For $t\in [0, t_{x,\mu})$, $\varphi_t^\mu x \in \RR^n$ and for $t\geq t_{x,\mu}$, $\varphi_t^\mu x= \Delta $.
\item[(ii)]  $\varphi_0^\mu x =x$ and for all $t,s\in [0, t_{x,\mu})$ with $t+s\in [0, t_{x,\mu})$, we have $t\in [0, t_{\varphi_s^\mu x,\mu})$ and $\varphi_{t+s}^\mu x= \varphi_t^\mu \varphi_s^\mu x$. 
\item[(iii)]  For all $t \in [0, t_{x,\mu})$ and $\epsilon >0$, there exists $\delta >0$ such that for all $(t', x', \mu') \in [0, \infty) \times \RR^n\times M$ with $$|t-t'|+|x-x'|+d_M(\mu,\mu')< \delta$$ we have $t' \in [0, t_{x',\mu'})$ and $$|\varphi_t^\mu x- \varphi_{t'}^{\mu'} x'| < \epsilon.$$ 
\end{itemize}
\end{definition}

Before proceeding further, we first provide an illustrative example that connects Definition~\ref{def:cls} with the examples in the sections above. 
\begin{example}
\label{ex:maincontrol}
Suppose that the system~\eqref{eqn:SDEp} is noise propagating and fix $Q\in \{ P_\text{L}, P_\text{D}\}$.  Consider the vector fields $X_0, X_1, \ldots, X_n$ on $\RR^n$ given by 
\begin{align}
X_0(x)=\Q(x) \qquad \text{ and }\qquad X_j (x) = e_j, \, \, j=1,2,\ldots, n,
\end{align}
where we recall that $e_1, \ldots, e_n$ is the standard orthonormal basis of $\RR^n$.  
Let $M=\RR^n$ with the usual Euclidean distance and recall that for fixed $x, \mu \in \RR^n$: 
 $$t \mapsto \varphi_t(X_0 + \textstyle{\sum_{j=1}^n} \sigma^j \mu^j X_j)x$$ is the maximally-defined (on $[0, \infty)$) solution of the ODE
\begin{align}
\label{eqn:controlextend}
\begin{cases}
\dot{x}_t= \Q(x_t)+ \sigma \mu \\
x_0= x.
\end{cases}
\end{align} 
Since the system is noise propagating and the equation~\eqref{eqn:controlextend} is explicitly integrable, we can set $t_{x,\mu}=\infty$ for all $x,\mu \in \RR^n$ since the time of explosion of~\eqref{eqn:controlextend} is infinite.  One can then check that the mapping $(t,x,\mu) \mapsto \varphi_t(X_0 + \textstyle{\sum_{j=1}^n} \sigma^j \mu^j X_j)x$ is a parameterized family of continuous local semigroups on $\RR^n$ parameterized by $M=\RR^n$.  Furthermore, to connect with the control problems in previous sections, for $t\in (0,1]$ and fixed $x,\mu\in \RR^n$, we note that $$\varphi_t(X_0 + \textstyle{\sum_{j=1}^n \sigma^j \mu^j X_j)}x=\varphi_t(\Q,f_\mu)x$$ where $f_\mu(t)=t\mu$ (see below equation~\eqref{eqn:controlgeneral}).  
\end{example}

\begin{remark}
Let $R\in \text{Poly}(\RR^n)$ be arbitrary and $M=\RR$.  Then the mapping
\begin{align*}
&(t,x, \alpha)\mapsto \varphi_t(\alpha R) x: [0, \infty) \times \RR^n \times \RR\rightarrow \RR^n \cup \{ \Delta \}
\end{align*}  
is a parameterized continuous local semigroup on $\RR^n$ parameterized by $\RR$.  Indeed, if $t_{x, \alpha}>0$ denotes the time of explosion for $t\mapsto \varphi_t(\alpha R)x$, then the conditions of Definition~\ref{def:cls} are met.  Also, in a similar way, the mapping
\begin{align*}
&(t,x) \mapsto \varphi_t(R)x: [0, \infty) \times \RR^n \rightarrow \RR^n \cup \{ \Delta\}
\end{align*}
can be viewed as a parameterized continuous local semigroup parameterized by $M=\RR$ by lifting the above map to 
\begin{align*}
&(t,x,\alpha) \mapsto \varphi_t(R) x: [0, \infty) \times \RR^n \times \RR\rightarrow \RR^n \cup \{ \Delta \}
\end{align*}
and such lift is implicitly assumed below. 
\end{remark}

Let $\mathcal{F}$ denote an arbitrary collection of parameterized continuous local semigroups on $\RR^n$.  A generic element of $\mathcal{F}$ is denoted by $\varphi$ and $\mathcal{P}(\varphi)$ denotes the parameter set of $\varphi$.  We sometimes use the notation $(\varphi, M)\in \mathcal{F}$ to indicate that $\varphi \in \mathcal{F}$ satisfies $\mathcal{P}(\varphi)=M$.  For $x\in \RR^n$ and $t>0$, let $\A_\mathcal{F}(x, t)$ be the set of points $y\in \RR^n$ such that there exist $\varphi^1, \varphi^2 \ldots, \varphi^k \in \mathcal{F}$, corresponding positive times $t_1, t_2, \ldots, t_k >0 $ for which $\sum_{j=1}^k t_j =t$ and parameters $\mu_1 \in \mathcal{P}(\varphi^1), \ldots, \mu_k \in \mathcal{P}(\varphi^k)$ such that   
\begin{align}
\label{eq:pcdt}
y = \varphi_{t_k}^{k, \mu_k} \varphi_{t_{k-1}}^{k-1,\mu_{k-1}} \cdots \varphi_{t_1}^{1, \mu_1} x.  
\end{align} 
Note that $\A_{\mathcal{F}}(x,t)$ is the set of points that can be accessed using the trajectories in $\mathcal{F}$ at time $t>0$ starting from $x\in \RR^n$.  Next, define 
\begin{align}
\A_\mathcal{F}(x, \leq t) = \bigcup_{0<s\leq t} \A_\mathcal{F}(x, s),
\end{align}
which is the set of points that can be accessed using the trajectories in $\mathcal{F}$ by time $t$ starting at $x$.  We call $\mathcal{F}$ \emph{approximately controllable} if for every $x\in \RR^n$ and $t>0$
\begin{align*}
\text{cl}(\A_\mathcal{F}(x,t))=\RR^n.
\end{align*}
The set $\mathcal{F}$ is called \emph{exactly controllable} if for every $x\in \RR^n$ and $t>0$
\begin{align*}
\A_\mathcal{F}(x,t)=\RR^n.
\end{align*}
\begin{remark}
It is possible to discuss control trajectories without the use of parameter sets.  However, parameter sets are useful as they allow one to uncover points in the sets $\A_\mathcal{F}(x,t)$, $x\in \RR^n$ and $t>0$, given some knowledge of points in their closures $\text{cl}(\A_\mathcal{F}(x,t))$, $x\in \RR^n$ and $t>0$. 
\end{remark}

Let $\mathcal{F}$ and $\mathcal{G}$ be two collections of parameterized continuous local semigroups on $\RR^n$. We say that $\mathcal{F}$ \emph{uniformly subsumes} $\mathcal{G}$, denoted by $\mathcal{G}\preccurlyeq_u \mathcal{F}$, if for any $\psi \in \mathcal{G}$, $t, \epsilon >0$ and any compact sets $D_1\subset \RR^n$, $D_2\subset \mathcal{P}(\psi)$ there exists $\varphi^1, \ldots, \varphi^k\in \mathcal{F}$, times $t_1, t_2, \ldots, t_k >0$ with $\sum t_j \leq t$ and continuous functions $f_k: D_2\rightarrow \mathcal{P}(\varphi^k)$ such that 
\begin{align*}
\sup_{x\in D_1, \mu \in D_2} \Big| \psi^\mu_t x - \varphi^{k, f_k(\mu)}_{t_k} \cdots \varphi^{1, f_1(\mu)}_{t_1} x \Big|<\epsilon. 
\end{align*}
 We say that $\mathcal{F}$ and $\mathcal{G}$ are \emph{uniformly equivalent}, denoted by $\mathcal{F}\sim_u \mathcal{G}$, if $\mathcal{G}\preccurlyeq_u \mathcal{F}$ and $\mathcal{F}\preccurlyeq_u \mathcal{G}$.  The \emph{uniform saturate} of $\mathcal{F}$, denoted by $\text{sat}_u(\mathcal{F})$, is defined by  
\begin{align}
\label{eqn:sateF}
\sat_u(\mathcal{F}) = \bigcup_{\mathcal{G} \preccurlyeq_u \mathcal{F}} \mathcal{G}. 
\end{align} 
One can show that (cf. \cite[Lemma 3.21]{GHHM_18}): $$\sat_u(\mathcal{F})\sim_u \mathcal{F}.$$  

The uniform saturate allows one to expand an initial collection $\mathcal{F}$ of trajectories to include simpler ones, from which control properties of $\mathcal{F}$ are more easily deduced.  Of course, by using this approach, one only gains access to sets of the form $\cl(\A_\mathcal{F}(x, \leq t))$, $x\in \RR^n$, $t>0$.  However, sometimes this information can be transferred back to the closures of exact time sets $\cl(\A_\mathcal{F}(x,t))$, $x\in \RR^n$, $t>0$.  This is the content of the following result, which is proved in~\cite[Lemma 3.13]{GHHM_18}.

\begin{lemma}[Conversion Lemma]
\label{lem:conversion}
Suppose $U \subset  \RR^n$ is open with the property that 
\begin{align*}
U \subset  \text{\emph{cl}}(\A_{\mathcal{F}}(x,\leq t))
\end{align*}
for all $x\in U$ and $t >0$.  Then 
\begin{align*}
U\subset  \text{\emph{cl}}(\A_{\mathcal{F}}(x,t))
\end{align*}  
for all $x\in U$ and $t >0$.  In particular, if $U=\RR^n$, then $\mathcal{F}$ is approximately controllable. 
\end{lemma} 

Next, once one characterizes points in the closed sets
\begin{align*} 
\text{cl}(\A_{\mathcal{F}}(x,t)), \,\,\, x\in \RR^n, \,\,\, t>0,
\end{align*}
it is natural to wonder which points in $\text{cl}(\A_{\mathcal{F}}(x,t))$ belong to $\A_{\mathcal{F}}(x,t)$.  This is the reason for the parameter set and the use of \emph{uniform} saturation above.  In essence, the uniformity in the above construction allows us to ``wiggle" the control slightly to go from approximate to exact controllability.  See~\cite[Section 3]{GHHM_18} for further details, which are partially summarized in Corollary~\ref{cor:exactcont} below. 

Given any nonempty $M\subset \RR^n$, we let $(\rho, M)$ denote the \emph{ray semigroup parameterized by} $M$; that is, for any $x\in \RR^n$, $t>0$ and $\mu \in M$:
\begin{align}
\label{eqn:raysemi}
\rho_t^\mu x := x+ t \mu.  
\end{align}  
As a special case of~\cite[Corollary 3.35]{GHHM_18}, we have:
\begin{corollary}
\label{cor:exactcont}
Let $\{ b_1, \ldots, b_n \}$ be any basis of $\RR^n$ and suppose $(\rho, \text{\emph{span}}(b_j)) \in \text{\emph{sat}}_u(\mathcal{F})$ for all $j$.  Then $\mathcal{F}$ is exactly controllable on $\RR^n$. 
\end{corollary}

Using Corollary~\ref{cor:exactcont}, the goal is to use scaling arguments to produce controllability results for $\mathcal{F}$. 

\subsection{Saturating trajectories}
In this section, we focus on a more concrete setting to maintain simplicity of discussion.  Fix $V_0 \in \text{Poly}(\RR^n)$ and $V_1, V_2, \ldots, V_k \in \text{Cons}(\RR^n)$, $k \leq n$, and for $(t,x, \mu) \in [0, \infty) \times \RR^n \times \RR^n$ let   
\begin{align}
\label{eqn:Fgen}
\varphi^\mu_t x=\varphi_t(V_0+\textstyle{\sum_{j=1}^k} \mu_j V_j) x.\end{align}
Note that $\varphi$ is a parameterized family of continuous local semigroups on $\RR^n$ parameterized by $M=\RR^k$ with the usual Euclidean distance.  We set 
\begin{align}
\label{eqn:Fdef}
\mathcal{F}(V_0; V_1, \ldots, V_k)= \{ \varphi\}. 
\end{align}

\begin{example}
\label{ex:comparetrajectories}
Recalling the setting of Example~\ref{ex:maincontrol}, consider $ \mathcal{F}(X_0; X_1, X_2, \ldots, X_n)$ and observe that a trajectory of the form~\eqref{eq:pcdt} corresponds to choosing $f\in \HH$ with piecewise constant derivative $\dot{f}$.  Consequently, 
\begin{align*}
\A_\mathcal{F}(x, t) \subset \A_t^2(x)\end{align*}
 for all $x\in \RR^n$ and $t\in (0,1] $.  Furthermore, if for $x\in \RR^n$ and $t\in (0,1]$ we define the set 
 \begin{align}
 \A^2_{\leq t}(x) = \{ y\in \RR^n \, : \, \varphi_s(x, f) =y \text{ for some }f\in \HH\text{ and } s\leq t\},
 \end{align}
 we also have that 
 \begin{align*}
 \A_\mathcal{F}(x,\leq t)\subset \A_{\leq t}^2(x).\end{align*}
 \end{example}
Our goal is to saturate the trajectories in $\mathcal{F}(V_0; V_1, \ldots, V_k)$ using two convenient scalings.  The first scaling utilizes the dynamics in~\eqref{eqn:Fgen} with large controls on small time intervals. 

\begin{proposition}
\label{prop:simplescaling}
Suppose that $D_1\subset \RR^n$ and $D_2\subset \RR$ are compact and $j\in \{1,2,\ldots, k \}$.  Then:
\begin{itemize}
\item[(i)] For any $t>0$, there exist $\lambda_*=\lambda_*(t)>0$ and $C_1(t)>0$ finite such that $\lambda \geq \lambda_*$  implies
\begin{align*}
\sup_{(s,x,\alpha ) \in [0,t/\lambda]\times D_1\times D_2} | \varphi_s(V_0+\lambda \alpha V_j) x| \leq C_1(t).  
\end{align*}
\item[(ii)] Let $v_j\in \RR^n$ be the constant value of $V_j\in \text{\emph{Cons}}(\RR^n)$.  For any $t>0$, there exists a constant $C_2(t)>0$ such that 
\begin{align*}
\sup_{(s, x, \alpha) \in [0, t]\times D_1 \times D_2} |\varphi_{s/\lambda}(V_0+\lambda \alpha V_j ) x - \rho_s^{\alpha v_j} x|\leq \frac{C_2(t)}{\lambda}
\end{align*}
for all $\lambda \geq \lambda_*$.  In particular, $(\rho, \text{\emph{span}}(v_j))\in  \sat_u(\mathcal{F}(V_0; V_1, \ldots, V_k))$.  
\end{itemize}
\end{proposition}

\begin{proof}
To prove part (i), fix $t >0$, $r>0$ and $\alpha_*>0$ such that $|x|\leq r$ for $x\in D_1$ and $|\alpha| \leq \alpha_*$ for $ \alpha \in D_2$.  For $(x, \alpha) \in D_1 \times D_2$, define $u_t^\lambda = \varphi_{t/\lambda}(V_0+\lambda \alpha V_j)x$. Note that 
\begin{align}
\begin{cases}
\frac{1}{2}\frac{d}{dt} |u_t^\lambda|^2 = \frac{1}{\lambda} V_0( u_t^\lambda)\cdot u_t^\lambda + \alpha v_j \cdot u_t^\lambda\\
u_0^\lambda=x.
\end{cases}
\end{align}
In particular, if $T_k^\lambda= \inf\{ t\geq 0 \, : |u_t^\lambda| \geq k\}$, for all $t< T_k^\lambda$ there exists a constant $C_k >0$ depending only on $k$ such that
\begin{align*}
\frac{1}{2}\frac{d}{dt}| u_t^\lambda|^2 \leq \frac{C_k}{\lambda} + |\alpha v_j| | u_t^\lambda|\leq \frac{C_k}{\lambda}+ \frac{|u_t^\lambda|^2}{2}+ \frac{\alpha^2|v_j|^2}{2}\leq \frac{C_k}{\lambda}+ \frac{|u_t^\lambda|^2}{2}+ \frac{\alpha_*^2|v_j|^2}{2}.\end{align*}  
Gr\"{o}nwall's inequality then implies that 
\begin{align*}
\sup_{s\in [0,T_k^\lambda \wedge t]} |u_s^\lambda|^2 \leq \Big(| x|^2 + \frac{2 C_k t}{\lambda}+ \alpha_*^2 |v_j|^2 t\Big) e^{t}\leq \Big(r^2+ \frac{2C_k t}{\lambda}+ \alpha_*^2 |v_j|^2 t\Big) e^{t}.\end{align*}
By first selecting $k=k(t)>0 $ large enough and then $\lambda_*=\lambda_*(t)>0$ large enough, we have the following bound satisfied for all $\lambda \geq \lambda_*$, $T_k^\lambda \geq t$:
\begin{align*}
\sup_{(s,x,\alpha)\in [0, t]\times D_1 \times D_2} |u_s^\lambda|^2 \leq \Big(r^2 + 1+ \alpha_*^2 |v_j|^2 t\Big) e^{t}< \infty.
\end{align*}
This proves part (i).

In order to prove part (ii), note that by using (i), for $\lambda \geq \lambda_*(t,x,j, \alpha)$ there exists a constant $C_2(t)>0$ such that 
\begin{align}
&\sup_{(s, x, \alpha) \in [0, t]\times D_1 \times D_2} |\varphi_{s/\lambda}(V_0+\lambda \alpha V_j ) x - \rho_s^{\alpha v_j } x|\\
\nonumber &=\sup_{(s, x, \alpha) \in [0, t]\times D_1 \times D_2}\bigg| \int_0^{s/\lambda} V_0(\varphi_u(V_0+\lambda \alpha V_j ) x) \, du\bigg|\leq \frac{C_2(t)}{\lambda}.  
\end{align}
This finishes the proof. 
\end{proof}  

To uncover even more trajectories in $\sat_u(\mathcal{F}(V_0; V_1, \ldots, V_k))$, fix $R\in \text{Poly}(\RR^n)$ and $V\in \text{Cons}(\RR^n)$ with $V(x)=v\in \RR^n$ for all $x\in \RR^n$.  For fixed $x\in \RR^n$, $\lambda \mapsto R(x+ \lambda v):\RR\rightarrow \RR^n$
is a vector of polynomials in $\lambda$.  Let $\deg(x, R, V)$ be the maximal degree among these polynomials and 
\begin{align}
\deg(R,V)= \max_{x\in \RR^n} \deg(x,R,V). 
\end{align}
The quantity $\deg(R,V)$ is called the \emph{relative degree} of $R$ and $V$.  Observe that Taylor's formula gives
\begin{align}
\label{eqn:reldegcomp}
R(x+\lambda v) =\sum_{|\eta| \leq \deg(R,V)} \frac{D^{\eta} R(x)}{\eta!} \lambda^{|\eta|} v^\eta,  
\end{align}
where we used multi-index notation in~\eqref{eqn:reldegcomp}.  Hence as $\lambda \rightarrow \infty$
\begin{align}
\label{eqn:bracketlim}
\frac{R(x+\lambda v)}{\lambda^{\deg(R,V)}} \longrightarrow \sum_{|\alpha| = \deg(R,V)}\frac{ D^\alpha R(x)}{\alpha!} v^\alpha 
\end{align}
where the convergence is uniform (in $x\in \RR^n$) on compact subsets of $\RR^n$.  

Next, observe that the quantity~\eqref{eqn:bracketlim} on the righthand side above can be expressed as a Lie bracket of vector fields.  Indeed, defining $\text{ad}^k V(R)$ inductively by $\text{ad}^0 V(R)=R$ and $\text{ad}^j V(R)= [V, \text{ad}^{j-1} V(R)]$ for $j\geq 0$, we find that 
\begin{align}
\label{eqn:bracketad}
 \sum_{|\alpha| = \deg(R,V)}\frac{ D^\alpha R(x)}{\alpha!} v^\alpha =  \frac{\text{ad}^{\deg(R,V)} V(R)}{|\deg(R,V)|!}(x)=: \Br(V,R)(x).\end{align}

\begin{example}
As an illustrative example, consider $R\in \text{Poly}(\RR^3)$ and $V\in \text{Cons}(\RR^3)$ given by 
\begin{align*}
R(x^1, x^2, x^3)= (x^1x^2, x^2 x^3, (x^3)^2) \qquad \text{ and } \qquad V(x^1, x^2, x^3)=(0,1,0).  
\end{align*}
Then $\deg (R,V) = 1$ since 
\begin{align*}
R(x^1, x^2+ \lambda, x^3) = (x^1x^2 + \lambda x^1, x^2 x^3+ \lambda x^3, (x^3)^2). 
\end{align*}
Furthermore, 
\begin{align*}
\text{Br}( V, R)(x^1, x^2, x^3)=\text{ad}^{\deg(R,V)}V (R)(x^1, x^2, x^3)= [V,R](x^1, x^2, x^3)= (x^1, x^3, 0). 
\end{align*}
\end{example}

The next goal is to see how $\text{Br}(V,R)$ arises using the relevant composition of trajectories.

\begin{proposition}
\label{prop:bracket}
Let $R\in \Poly(\RR^n)$ and $V\in \Cons(\RR^n)$ with constant value $v\in \RR^n$.  Let $D_1\subset \RR^n$ and $D_2\subset \RR$ be compact, and fix $t>0$ such that 
$\varphi_t(\alpha^{\deg(R,V)}  \text{\emph{Br}}(V,R)) x\in \RR^n$ for all $(x,\alpha) \in D_1 \times D_2$.  Then 
\begin{align*}
\sup_{(s,x,\alpha)\in [0,t]\times D_1 \times D_2 } | \rho_{1/\lambda}^{-\lambda^2 \alpha v} \varphi_{\frac{s}{\lambda^{\deg(V,R)}}} (R) \rho_{1/\lambda}^{\lambda^2 \alpha v} x -\varphi_s(\alpha^{\deg(V,R)} \text{\emph{Br}}(V,R))x | \rightarrow 0
\end{align*}     
as $\lambda \rightarrow \infty$.  In particular, if 
\begin{align*}
(t, x) \mapsto \varphi_t(R) \in \text{\emph{sat}}_u(\mathcal{F}(V_0; V_1, \ldots, V_k))\qquad \text{ and } \qquad (\rho, \text{\emph{span}}(v)) \in \text{\emph{sat}}_u(\mathcal{F}(V_0; V_1, \ldots, V_k)),
\end{align*}
then $(t, x, \alpha) \mapsto \varphi_t(\alpha^{\text{\emph{deg}}(R,V)} \text{\emph{Br}}(V,R))x \in\text{\emph{sat}}_u(\mathcal{F}(V_0; V_1, \ldots, V_k))$.   
\end{proposition}
%

\begin{proof}
Let $N= \deg(R,V)$ and $R'=\Br(V,R)$.  For $(x,\alpha) \in D_1\times D_2$ and $s\leq t$ set
\begin{align}
w_s^\lambda = \rho_{\frac{1}{\lambda}}^{-\lambda^2 \alpha v} \, \varphi_{\frac{s}{\lambda^N}} (R)\,\rho_{\frac{1}{\lambda}}^{\alpha \lambda^2 v} x. 
\end{align}
Since $(s,x, \alpha) \mapsto \varphi_s(\alpha^{\deg(R,V)}  \text{Br}(V,R)) x$ is a continuous mapping on the compact set $[0, t]\times D_1\times D_2$ into $\RR^n$, there is $j>0$ large enough so that 
\begin{align*}
\sup_{(s,x, \alpha)\in [0,t]\times D_1\times D_2} | \varphi_s(\alpha^{\deg(R,V)}  \text{Br}(V,R)) x|< j.\end{align*}
For $\lambda >0$ introduce
\begin{align*}
S_{2j}^{\lambda}= \inf\{ r \geq 0\, : \, |w_r^\lambda|\geq 2 j \}.
\end{align*}
Note that for any $s< S_{2j}^{\lambda}\wedge t$ and $(x, \alpha) \in D_1\times D_2$ we have   
\begin{align*}
\frac{1}{2}\frac{d}{ds} | w_s^\lambda - \varphi_s(\alpha^N R') x|^2 &= [\lambda^{-N} R(w_s^\lambda+\lambda \alpha v) - \alpha^N R'(\varphi_s(\alpha^N R') x)]\cdot (w_s^\lambda - \varphi_s(\alpha^N R') x).
\end{align*}
Thus applying~\eqref{eqn:reldegcomp} to the term $R(w_t^\lambda+ \lambda \alpha v)$, we obtain for $s< S_{2j}^\lambda\wedge t$ and $(x,\alpha) \in D_1 \times D_2$:  
\begin{align*}
\frac{1}{2}\frac{d}{ds} | w_s^\lambda - \varphi_s(\alpha^N R') x|^2 &=[\alpha^N R'(w_s^\lambda) - \alpha^N R'(\varphi_s(\alpha^N R') x)]\cdot (w_s^\lambda - \varphi_s(\alpha^N R') x) \\
\\ \qquad &+ \frac{1}{\lambda^N}\bigg(\sum_{|\eta| \leq N-1} \frac{D^\eta R(w_s^\lambda)}{\eta !} (\lambda \alpha)^{|\eta|} v^\eta \bigg) \cdot  (w_s^\lambda - \varphi_s(\alpha^N R') x)\\
& \leq C_{j} \frac{|w_s^\lambda - \varphi_s(\alpha^N R') x|^2}{2} + \frac{K_{j}}{2\lambda},
\end{align*}
where $C_{j}, K_{j} >0$ are constants depending only on $j$.  Then Gr\"{o}nwall's inequality gives that 
\begin{align*}
\sup_{(s,x, \alpha)\in [0, S_{2j}^\lambda \wedge t] \times D_1 \times D_2}| w_s^\lambda - \varphi_s(\alpha^N R') x|^2 \leq \frac{K_{j} t}{\lambda} e^{C_{j} t}.
\end{align*}
Consequently, $\liminf_{\lambda \rightarrow \infty}S_{2j}^\lambda > t$ and the result follows.  
\end{proof}

Proposition~\ref{prop:simplescaling} and Proposition~\ref{prop:bracket} give rise to an inductive procedure that can be used to produce trajectories in $\sat_u(\mathcal{F})$, which we now outline.  Note that this is the same procedure as in~\cite{HM_15}, but in the language of parameterized semigroups.   

For any collection $\mathcal{G}$ of parameterized continuous local semigroups on $\RR^n$, we call $V\in \text{Cons}(\RR^n)$ \emph{scalable for} $\mathcal{G}$ if $(\rho, \text{span}(v))\in \mathcal{G}$ where $v\in \RR^n$ is the constant value of $V$ and $(\rho, M)$ was defined in~\eqref{eqn:raysemi}. The set of constant vector fields which are scalable for $\mathcal{G}$ is denoted by $\text{Scale}(\mathcal{G})$.   Recalling $\mathcal{F}(V_0; V_1, \ldots, V_k)$ given in~\eqref{eqn:Fdef}, we define $\mathcal{F}_0=\mathcal{F}_0(V_0; V_1, \ldots, V_k)$ and $\mathcal{F}_1=\mathcal{F}_1(V_0; V_1, \ldots, V_k)$ by  
\begin{align*}  
\mathcal{F}_0&= \mathcal{F}(V_0; V_1, \ldots, V_k), \\
\mathcal{F}_1&= \mathcal{F}_0 \cup \{ (\rho, \text{span}(v_j) ) \, : \, j=1,2,\ldots, k, \,\, V_j(x)=v_j \in \RR^n\text{ for all } x\in \RR^n \}.
\end{align*}
For $j\geq 2$, we define $\mathcal{F}_j =\mathcal{F}_j(V_0; V_1, \ldots, V_k)$ inductively by 
\begin{align*}
\mathcal{F}_j&= \mathcal{F}_{j-1} \cup \\
& \{(t,x, \alpha) \mapsto \varphi_t(\alpha^{\text{deg}(Q,V)}\Br(V,Q))x \, : \, V\in \text{Scale}(\mathcal{F}_{j-1}), (t,x) \mapsto \varphi_t(Q)x \in \mathcal{F}_{j-1}, Q\in \text{Poly}(\RR^n)\}.  
\end{align*} 
Observe that by combining Proposition~\ref{prop:simplescaling} and Proposition~\ref{prop:bracket}, we have that 
\begin{align}
\bigcup_{j=0}^\infty \mathcal{F}_j(V_0; V_1, \ldots, V_k) \subset  \sat(\mathcal{F}_u(V_0; V_1, \ldots, V_k)). 
\end{align}

\subsubsection{Examples}  In this section, we apply the previous results concerning the uniform saturate.  
\begin{example}
In this example, we present a simple scenario to illustrate the basic idea of the results in the previous section. Fix $a,b\in \RR$ and let $V_0\in \text{Poly}(\RR^2)$ and $V_1\in \text{Cons}(\RR^2)$ be given by 
\begin{align*}
V_0(x,y)=(x^2-ay^2+by, 2x) \qquad \text{ and } \qquad V_1(x,y)=e_2.
\end{align*} 
Define $\mathcal{F}= \mathcal{F}(V_0; V_1)$ and $\mathcal{F}_j= \mathcal{F}_j(V_0; V_1)$, $j\geq 0$, as above. 
Note that if $v_1$ is the constant value of $V_1\in \text{Cons}(\RR^n)$, then 
\begin{align*}
(\rho, \text{span}(v_1))\in \mathcal{F}_1 ; 
\end{align*}
that is, $V_1\in \text{Scale}(\mathcal{F}_1)$.  Next, observe that $\deg(V_1, V_0)=2$ if $a\neq 0$ and $\deg(V_1, V_0)=1$ if $a=0, b\neq 0$.  Hence, 
\begin{align} 
\Br( V_1, V_0)(x,y) = \begin{cases}
-a e_1& \text{ if } a\neq 0 \\
 b e_1& \text{ if } a=0, b\neq 0.
\end{cases}
\end{align}  
Thus if $a\neq 0$, then $(\rho, \text{cone}(-a e_1)) \in \mathcal{F}_2$, where 
\begin{align*}
\text{cone}(-a e_1)= \{ (-a\lambda, 0) \, : \, \lambda \geq 0 \}. 
\end{align*}
Alternatively if $a=0, b\neq 0$, then $(\rho, \text{span}(e_1) ) \in \mathcal{F}_2$.  Consequently, if $a\neq 0$ it follows that for every $(x,y) \in \RR^2$ and every $t>0$
\begin{align*}
 \overline{\A_\mathcal{F}((x,y), \leq t)}\supset \{ (x, y) - \lambda a e_1 + \alpha e_2 \, : \, \lambda \geq 0, \alpha \in \RR \} .
\end{align*} 
On the other hand, if $a=0$ and $b\neq 0$, then using Corollary~\ref{cor:exactcont} we see that for every $(x,y) \in \RR^2$ and every $t>0$, 
\begin{align*}
 A_\mathcal{F}(x, t)=\RR^2.
\end{align*}

\end{example}

\begin{example}\label{ex:IK}

Let $V_1\in \text{Poly}(\RR^n)$ be given by $V_1(x)=e_1$ and $V_0\in \text{Cons}(\RR^n)$ satisfy
\begin{align*}
V_0^j(x) = \begin{cases}
0 & \text{ if } j=1\\
x^{j-1}& \text{ if } j=2,3,\ldots, n.
\end{cases}
\end{align*} 
Define $\mathcal{F}= \mathcal{F}(V_0; V_1)$ and $\mathcal{F}_j= \mathcal{F}_j(V_0; V_1)$, $j\geq 0$, as above. 
Then $(\rho, \text{span}(e_1)) \in \mathcal{F}_1$.  Also, $\deg(V_1, V_0)=1$ and $\text{Br}(V_1, V_0) = \partial_{x^2}$, so $(\rho, \text{span}(e_2) ) \in \mathcal{F}_2$ and $V_2\in \text{Cons}(\RR^n)$ given by $V_2(x)=e_2$ is such that $V_2 \in \text{Scale}(\mathcal{F}_2)$.  Inductively, $(\rho, \text{span}(e_j)) \in \mathcal{F}_j$ for all $j$, so 
by Corollary~\ref{cor:exactcont} we have  
\begin{align*}
\A_\mathcal{F}(x,t) = \RR^n 
\end{align*}
for all $x\in \RR^n, t>0$. 
\end{example}

\begin{example}
\label{ex:controllor96}
Recalling the context of the scaling limits for the Lorenz '96 model discussed in Example~\ref{ex:lor96} and Example~\ref{ex:lor92part2}, let $n\geq 4$, $V_0 \in \Poly(\RR^n)$ and $V_1, V_2 \in \Cons(\RR^n)$ be given by
\begin{align*}
V_0^j(x) = \begin{cases}
0 & \text{ if } j=1,2\\
-x^{j-2} x^{j-1} & \text{ if } j=3,4,\ldots, n-1\\
x^1 x^{n-1} & \text{ if } j=n
\end{cases}\qquad \text{ and } \qquad V_1(x)=e_1,\,\, V_2(x)=e_2.
\end{align*}
Let $\mathcal{F}=\mathcal{F}(V_0; V_1, V_0)$ and $\mathcal{F}_j = \mathcal{F}_j(V_0; V_1, V_2)$, $j\geq 0$, be as above.  Then $(\rho, \text{span}(e_j)) \in \mathcal{F}_1$ for $j=1,2$.  Observe that $\deg(V_1, V_0)=1$ and 
\begin{align*}
\Br(V_1, V_0)= [V_1, V_0] = - x^2 \partial_{x^3}+ x^{n-1} \partial_{x^n}.  
\end{align*}   
It follows from Proposition~\ref{prop:bracket} that $(t,x) \mapsto \varphi_t(\alpha \Br(V_1, V_0))x \in \mathcal{F}_2$ for every $\alpha\in \RR$.  Furthermore, $\deg(V_2, \Br(V_1, V_0))=1$ and 
\begin{align*}
\Br(V_2, \Br(V_1, V_0)) = [ V_2, [V_1, V_0]] =  - \partial_{x_3}.  
\end{align*}
Hence, $(\rho, \text{span}(e_3)) \in \mathcal{F}_3$ so that if $V_3=\partial_{x^3}$, $V_3\in \text{Scale}(V_3)$.  Inductively, if we define $V_j = e_j$ for $j=4,5, \ldots, n$, then there exists $m$ such that $(\rho, \text{span}(e_j)) \in \mathcal{F}_m$ for $j=4,5,\ldots, n$ and all $\alpha \in \RR$.  Hence, applying Corollary~\ref{cor:exactcont}
\begin{align*}
 \A_\mathcal{F}(x, t)=\RR^n. 
\end{align*}   
for all $x\in \RR^n$ and $t>0$. 

Next, still in the context of the Lorenz '96 model, our goal is to show that for any $t\in (0,1]$, there exists $r>0$ such that 
\begin{align}
\label{eq:idaos}
\text{cl}(A_{\text{L},t}(0)) \supset B_r(0).  
\end{align} 
When considering the LIL scaling $x_{\epsilon,t}$ as in~\eqref{eqn:LILrescaling}, then~\eqref{eq:idaos} implies that, for fixed $t$, the collection of limit points as $\epsilon \rightarrow 0^+$ of the $\RR^n$-valued random variables $\{ x_{\epsilon, t} \}_{\epsilon}$ contains $\text{cl}(B_r(0))$.  See the discussion around~\eqref{eqn:limitptfixed} for more details.

To prove~\eqref{eq:idaos},  we note that by Lemma~\ref{lem:comphom} the control problem~\eqref{eqn:controlLIL} is component homogeneous since each component of $P_\text{L}$ is a monomial.  Furthermore, since $A_{\text{D}, t}(x) \supset A_\mathcal{F}(x,t)=\RR^n$ for $x\in \RR^n, t\in (0,1]$, then $0\in A_{\text{D}, t}(x)$ for all $x\in \RR^n$, $t\in (0,1]$.  Thus ~\eqref{eq:idaos} follows from Theorem~\ref{thm:transfer}.  
\end{example}

\section{practical criteria for regular points}
\label{sec:regularpt}

Throughout this section, $\mathscr{O}\subset  \RR^n$ denotes a non-empty open set with continuous boundary $\partial \mathscr{O}$ containing $0\in \RR^n$.  We recall that $x_t$ denotes the solution of~\eqref{eqn:SDEp} and that $0\in \RR^n$ is called \emph{regular} for $(x_t, \mathscr{O})$ if $\PP\{ \xi=0\}=1$, where
\begin{align}
\xi = \inf\{ t>0 \, : \, x_t \in \mathscr{O} \}.
\end{align}
We furthermore suppose there exists $j\in \{1,2,\ldots, n\}$, $\delta >0$ and a continuous function $b:\RR^{n-1}\rightarrow \RR$ for which
\begin{align}
\label{eqn:domain1}
\partial \mathscr{O} \cap B_\delta(0) &= \{ x\in B_\delta(0) \, : \, x^j = b(x^1, \ldots, x^{j-1}, x^j,\ldots x^{n})\}, \quad \text{ and } \\
\label{eqn:domain2}\mathscr{O} \cap B_\delta(0)&=   \{ x\in B_\delta(0) \, : \, x^j > b(x^1, \ldots, x^{j-1}, x^j,\ldots x^{n})\}.\end{align}
The goal of this section is to provide general criteria for $0\in \RR^n$ to be regular for $(x_t, \mathscr{O})$.  We then apply the criteria in several examples.

In order to characterize when $0\in \RR^n$ is regular for $(x_t,\mathscr{O})$, we use the distributional scaling~\eqref{eqn:rescaledist}.  Specifically, supposing that \eqref{eqn:SDEp} is noise propagating, introduce a transformation $S_{\epsilon, \text{D}} : \RR^n\rightarrow \RR^n$, $\epsilon >0$, given by 
\begin{align}
S_{\epsilon, \text{D}}(x^1, x^2, \ldots, x^n) = \Big( \frac{x^1}{\epsilon^{\mathfrak{b}_1}},  \frac{x^2}{\epsilon^{\mathfrak{b}_2}}, \ldots,  \frac{x^n}{\epsilon^{\mathfrak{b}_n}}\Big).\end{align} 
Let $r_\epsilon$, $\epsilon >0$, be the one-parameter family of open rectangles defined by
\begin{align}
r_\epsilon = (-\epsilon^{\mathfrak{b}_1}, \epsilon^{\mathfrak{b}_1}) \times(-\epsilon^{\mathfrak{b}_2}, \epsilon^{\mathfrak{b}_2}) \times \cdots \times (-\epsilon^{\mathfrak{b}_n}, \epsilon^{\mathfrak{b}_n}), \,\, \epsilon >0,\end{align} 
and observe that $S_{\epsilon, \text{D}}$ maps $r_\epsilon$ bijectively onto $r_1=(-1,1)^n$.  By our hypothesis on the domain $\mathscr{O}$, there exists $\epsilon_*>0$ such that 
\begin{align}
\Gamma_\epsilon:=\{ x\in r_\epsilon \, : \, x^j = b(x^1, \ldots, x^{j-1}, x^j,\ldots x^{n})\} &= r_\epsilon \cap \partial \mathscr{O},\\
\mathscr{O}_{\epsilon}:=   \{ x\in r_\epsilon \, : \, x^j > b(x^1, \ldots, x^{j-1}, x^j,\ldots x^{n})\} &= r_\epsilon \cap \mathscr{O},\end{align}
for every $\epsilon \in (0, \epsilon_*)$.  Also, by definition of the transformation $S_{\epsilon, \text{D}}$, observe that for any $\epsilon \in (0, \epsilon_*)$
\begin{align*}
S_{\epsilon, \text{D}}(\Gamma_\epsilon) = \{ y\in(-1,1)^n \, :\, \epsilon^{\mathfrak{b}_j}y^j = b(y^1 \epsilon^{\mathfrak{b}_1}, \ldots, y^n \epsilon^{\mathfrak{b}_n})\},\\
S_{\epsilon, \text{D}} (\mathscr{O}_{\epsilon}) =\{ y\in (-1,1)^n \, :\, \epsilon^{\mathfrak{b}_j}y^j > b(y^1 \epsilon^{\mathfrak{b}_1}, \ldots, y^n \epsilon^{\mathfrak{b}_n}) \}.
\end{align*}
Recalling the accessibility set $A_{\text{D}, t}(0)$ introduced in~\eqref{def:AD}, we have the following:

\begin{theorem}[Regular Point Criteria]
\label{thm:criteriareg}
Assume that \eqref{eqn:SDEp} is noise propagating.  If there exists a non-empty open set $\mathscr{O}_*$ for which  
\begin{align*}
\mathscr{O}_* \subset  S_{\epsilon,\emph{\text{D}}}(\mathscr{O}_{\epsilon})   \text{ for all } \epsilon \in (0, \epsilon_*) \,\,\,\text{ and }\,\,\, \bigcup_{t\in (0, 1]} \cl(\A_{\text{\emph{D}},t}(0)) \cap \mathscr{O}_* \neq \emptyset ,
\end{align*} 
then $0\in \RR^n$ is regular for $\mathscr{O}$. 
\end{theorem}
\begin{remark}
The hueristics behind Theorem~\ref{thm:criteriareg} is that the distributional rescaling gives the local asymptotic behavior of the process $x_t$ solving \eqref{eqn:SDEp} at time zero in the distributional sense.  As long as the domain $\mathscr{O}$ ``plays well" with the underlying scaling $S_{\epsilon, \text{D}}$ and the resulting limiting control problem visits points in $\mathscr{O}_*$, then $0\in \RR^n$ is regular for $(x_t, \mathscr{O})$.    
\end{remark}
\begin{proof}[Proof of Theorem~\ref{thm:criteriareg}]
Recalling the process $t\mapsto y_{\epsilon, t}$ given in~\eqref{eqn:Deps}, observe that for any $t>0$ and $\epsilon >0$ small enough
\begin{align*}
\PP \{ \xi \leq \epsilon t\} &=\PP\big\{x_{\epsilon s} \in \mathscr{O} \text{ for some } s\in (0,t] \big\}\\
&\geq \PP \big\{ y_{\epsilon,s} \in S_{\epsilon, \text{D}}(\mathscr{O}_{\epsilon})\text{ for some } s\in (0, t] \big\}\\
&\geq \PP \{ y_{\epsilon,s} \in \mathscr{O}_* \text{ for some } s\in (0, t] \} \geq \PP \{ y_{\epsilon,t} \in \mathscr{O}_* \}.
\end{align*}  
By Corollary~\ref{cor:dist}, $\PP \{ y_{\epsilon,t} \in \mathscr{O}_* \} \rightarrow \PP\{ y_t \in \mathscr{O}_*\}$ as $\epsilon \rightarrow 0$ where $y_t$ solves~\eqref{eqn:SDEdist}.  Since $\cl(\A_{\text{D},t}(0)) \cap \mathscr{O}_* \neq \emptyset$ for some $t\in (0,1]$ and $\mathscr{O}_*$ is open, Lemma~\ref{lem:support} implies that for some $t \in (0,1]$, $\PP\{ y_t \in \mathscr{O}_* \} >0$.  In particular, we have 
\begin{align*}
\PP \{ \xi =0 \} \geq \PP\{ y_t \in \mathscr{O}_* \} >0.  
\end{align*}
By Blumenthal's $0-1$ law, it follows that $\PP\{ \xi =0 \}=1$, finishing the proof. 
\end{proof}

%
%

We next consider a few examples to illustrate Theorem~\ref{thm:criteriareg}.    

\begin{example}
\label{ex:BMsimple}
As a preliminary example, we first consider the simple case of a standard Browian motion on $\RR^n$, $n\geq 2$.  That is, suppose in equation~\eqref{eqn:SDEp}, $\sigma = I_{n\times n}$ and $P(x)\equiv 0$ so that $x_t= B_t$.  The associated scaling transformation $S_{\epsilon, \text{D}}$ is given by$$S_{\epsilon, \text{D}}(x)= \frac{x}{\epsilon^{1/2}}.$$
Furthermore, the associated control problem is given by 
\begin{align}
\begin{cases}
\dot{x}= \dot{f}\\
x_0=0
\end{cases}
\end{align}
where the control $f$ belongs to the class $\mathscr{H}$.  Note that, because one has control in all directions, it follows that $\A_{\text{D},t}(0)=\RR^n$ for all $t\in (0,1]$.

Next, consider any domain $\mathscr{O} \subset  \RR^2$ such that, for some $\alpha  >0$ and some $\delta >0$, we have 
\begin{align*}
\partial \mathscr{O} \cap B_\delta(0) = \{ x\in B_\delta(0) \, : \, x^n = \alpha |(x^1, \ldots, x^{n-1})| \}, \\
 \mathscr{O} \cap B_\delta(0) = \{ x\in B_\delta(0) \, : \, x^n > \alpha  | (x^1, \ldots, x^{n-1})| \}.
\end{align*} 
Since the cone $x^n= \alpha |(x^1, \ldots, x^{n-1})|$ is invariant under $S_{\epsilon, \text{D}}$ for any $\epsilon >0$, there exists $\epsilon_*>0$ for which 
\begin{align*}
S_{\epsilon, \text{D}} (\mathscr{O}_{\epsilon})&= \{ x\in r_1 \, : \, x^n > \alpha |(x^1, \ldots, x^{n-1})|\} \,\,\text{ for all } \epsilon \in (0,\epsilon_*).
\end{align*}
As observed above, $\A_{\text{D},t}(0) = \RR^2$ for any $t\in (0,1]$.  Thus it follows that $0\in \RR^2$ is regular for $(x_t, \mathscr{O})$ by Theorem~\ref{thm:criteriareg}.    
\end{example}

\begin{example}
\label{ex:IKsimple}
Suppose that $x_t=(x_t^1, \ldots, x_t^n)$ satisfies
\begin{align}
\begin{cases}
dx_t^1 = dB^1_t \\
dx_t^2 = x_t^1 \, dt \\
\quad \vdots \qquad \vdots \\
dx_t^{n}= x_t^{n-1} \, dt\\
x_0 = 0.
\end{cases}
\end{align}
This example was studied extensively in the papers~\cite{Lac_97ii, Lac_97i}.  In particular, both the law of the iterated logarithm (non-functional)~\cite{Lac_97ii} as well as a necessary and sufficient condition for $0\in \RR^n$ to be regular for $(x_t, \mathscr{O})$ were obtained~\cite{Lac_97i}.  There, the criteria for $0\in \RR^n$ to be regular goes beyond the scope of the results in this work, in particular the types of permissible boundaries.  Here we show how this example fits into our framework.

In this case, $x_t$ is noise propagating and the associated distributional scaling transformation $S_{\epsilon, \text{D}}$ is given by 
\begin{align*}
S_{\epsilon, \text{D}}(x^1, \ldots, x^n)=\Big(\frac{x^1}{\epsilon^{1/2}}, \frac{x^2}{\epsilon^{3/2}}, \ldots, \frac{x^{n}}{\epsilon^{(2n-1)/2}}\Big).\end{align*}
Let $\mathscr{O}\subset  \RR^n$ be any domain satisfying~\eqref{eqn:domain1} and~\eqref{eqn:domain2} for some $j$
 and suppose furthermore that there exists $\epsilon_*>0$ and non-empty open $\mathscr{O}_* $ with $\mathscr{O}_*\subset S_{\epsilon, \text{D}}(\mathscr{O}_\epsilon)$ for all $\epsilon \in (0, \epsilon_*)$.  As a concrete example, one could take $j=1$ and 
 \begin{align*}
 b(x^2, x^3, \ldots, x^n)= |x^2|^{1/3}+ |x^3|^{1/5}+ \cdots + |x^n|^{1/(2n-1)}.  
 \end{align*} 
 Also, note that $\A_{\text{D},t}(0)=\RR^n$ for all $t\in (0,1]$ by Example~\ref{ex:IK} .  Hence, it follows that $0\in \RR^n$ is regular for $(x_t, \mathscr{O})$. 
 \end{example}

\begin{example}
As a nontrivial example, consider the stochastic Lorenz '96 model introduced in Example~\ref{ex:lor96}.  That is, let $x_t$ solve~\eqref{eqn:stochlor1} and recall that $x_t$ is noise propagating.  Example~\ref{ex:lor92part2} shows that $P_\text{L}=P_\text{D}$, where $P_\text{L}$ is given explicitly in~\eqref{eqn:PLlor}.  Also, $\mathfrak{b}_j= (\mathfrak{a}_j^1, 0)$, where $\mathfrak{a}_j=(\mathfrak{a}_j^1, \mathfrak{a}_j^2)$ satisfies $\mathfrak{a}_1=\mathfrak{a}_2=(1/2,1/2)$ with 
\begin{align*}
\mathfrak{a}_j = \mathfrak{a}_{j-1}+ \mathfrak{a}_{j-2}+ (1,0), \,\,\, j=3,4,\ldots, n-1,
\end{align*} 
and $\mathfrak{a}_n= (3/2,1/2)+ \mathfrak{a}_{n-1}$.  Furthermore, we saw in Example~\ref{ex:controllor96} that $A_{\text{D}, t}(0)=\RR^n$ for all $t\in (0,1]$.  Thus any domain $\mathscr{O}$ for which \eqref{eqn:domain1} and~\eqref{eqn:domain2} are satisfied for some $j$ and such that there exists a nonempty open set $\mathscr{O}_*\subset S_{\epsilon, \text{D}}(\mathscr{O}_\epsilon)$ for all $\epsilon >0$ small enough, $0\in \RR^n$ is regular for $(x_t, \mathscr{O})$.   
\end{example}

\subsection{General remarks and second-order Langevin dynamics}
Suppose now that $\mathscr{U}\subset \RR^n$ is a non-empty open set with continuous boundary $\partial \mathscr{U}$ which possibly does not contain the origin.  If $x_t$ denotes the solution of~\eqref{eqn:SDEp} and $x_* \in \RR^n$, we let $x_t(x_*):=x_t+x_*$ and observe that $x_t(x_*)$ satisfies
\begin{align}
\label{eqn:SDEpshift}
\begin{cases}
dx_t(x_*)= P(x_t(x_*)-x_*) \, dt + \sigma dB_t\\
x_t(x_*)=x_*. 
\end{cases}
\end{align}
Note that \eqref{eqn:SDEpshift} has the same generic form as \eqref{eqn:SDEp} but the resulting equation is started from a general initial condition $x_*$.  For $x_* \in \partial \mathscr{U}$, let 
\begin{align}
\xi_{x_*} = \inf\{ t >0 \,: \, x_t(x_*) \in \mathscr{U}\}.  
\end{align}
We say that $x_* \in \partial \mathscr{U}$ is \emph{regular} for $(x_t(x_*), \mathscr{U})$ if $\PP\{ \xi_{x_*} < \infty\} =1$ and \emph{irregular} for $(x_t(x_*), \mathscr{U})$ otherwise.  We first state a result that follows directly from the definitions.   
\begin{proposition}
\label{prop:regequiv}
Let $x_* \in \partial \mathscr{U}$ and $\mathscr{O}:= \{ x\in \RR^n \, : \, x = y- x_* \text{ for some } y\in \mathscr{U}  \}$.  Then $x_* $ is regular for $(x_t(x_*), \mathscr{U})$ if and only if $0\in \partial \mathscr{O}$ is regular for $(x_t, \mathscr{O})$.    
\end{proposition}

Our goal is to characterize regular points for second-order Langevin dynamics started on level sets of the underlying Hamiltonian.  That is, let $U\in C^2(\RR^k; [0, \infty))$ and consider
\begin{align}
\label{eqn:LDG}
\begin{cases}
dq_t = p_t\, dt \\
dp_t = -p_t\, dt - \nabla U(q_t) \, dt + \sqrt{2} \, dB_t \\
(q_0, p_0)\in \RR^n 
\end{cases}\end{align}
where $B_t$ is a standard, $k$-dimensional Brownian motion.  
Recall that this is the same context as discussed in Example~\ref{ex:LD2}, but here we also assume that $U\geq 0$.  However, since only the gradient of $U$ appears in~\eqref{eqn:LDG}, we can always add a constant to $U$ to achieve $U\geq 0$.  See also Remark~\ref{rem:LD} which shows that we do not need to assume $U$ is a polynomial to apply the scaling results in the paper. 

Fix $R>0$ and let 
\begin{align}
H(q,p) = \frac{|p|^2}{2}+ U(q)
\end{align}
 be the Hamiltonian of the system~\eqref{eqn:LDG}.  We study the regular point problem for the process $(q_t, p_t)$ started on 
\begin{align}
\partial H_{>R} := \{(q,p) \in \RR^k \times \RR^k \,: \,H(q,p)=R\}
\end{align}
which is the boundary of the open set  
\begin{align}
H_{>R}= \{ (q,p) \in \RR^k \times \RR^k \,: \,H(q,p) > R\}. 
\end{align}  
We assume below that $R>0$ is chosen so that both $H_{>R}$ and $\partial H_{>R}$ are non-empty.  
\begin{theorem}
Every $(q_0, p_0) \in \partial H_{>R}$ is regular for $((q_t, p_t), H_{>R})$. 
\end{theorem}

\begin{proof}
Let $(q_0, p_0) \in \partial H_{>R}$ and consider the shifted sets
\begin{align}
\mathscr{O}_{(q_0, p_0)} &= \{ (\overline{q},\overline{p}) \in \RR^k \times \RR^k \, : \,H(\overline{q}+q_0, \overline{p}+p_0)>R \},\\
\partial \mathscr{O}_{(q_0, p_0)}&=\{ (\overline{q},\overline{p}) \in \RR^k \times \RR^k \, : \,H(\overline{q}+q_0, \overline{p}+p_0) =R\}.
\end{align}
Similarly, we consider the shifted process $(\overline{q}_t, \overline{p}_t):=( q_t-q_0, p_t-p_0)$ which satisfies 
\begin{align}
\label{eqn:LDNP}
\begin{cases}
d\overline{q}_t= (\overline{p}_t +p_0)\, dt \\
d\overline{p}_t = - (\overline{p}_t+p_0)\, dt - \nabla U(\overline{q}_t+q_0) \, dt + \sqrt{2} \, dB_t\\
(\overline{q}_0, \overline{p}_0)=0.
\end{cases}
\end{align}
We recall from Example \ref{ex:LD2} and Remark~\ref{rem:LD} that~\eqref{eqn:LDNP} is noise propagating.  Furthermore, it follows that $\mathfrak{b}_{k+1}=\mathfrak{b}_{k+2}=\cdots =\mathfrak{b}_{2k}= 1/2$ and, for $j=1,2,\ldots, k$, 
\begin{align}
\mathfrak{b}_j = \begin{cases}
1 &\text{ if } p_0^j \neq 0 \\
\frac{3}{2}& \text{ if } p_0^j =0. 
\end{cases}
\end{align}
Also, $P_\text{D}^j(\overline{q},\overline{p}) =0$ if $j=k+1, \ldots, 2k$ and 
\begin{align}
P_\text{D}^j(\overline{q},\overline{p}) = \begin{cases}
p_0^j & \text{ if } p_0^j \neq 0 \\
\overline{p}^j & \text{ if } p_0^j =0.
\end{cases}
\end{align}
Next, let $\epsilon >0$ and $r_\epsilon =(-\epsilon^{\mathfrak{b}_1}, \epsilon^{\mathfrak{b}_2})\times \cdots \times( -\epsilon^{\mathfrak{b}_{2k}}, \epsilon^{\mathfrak{b}_{2k}})$.  Consider 
\begin{align*}
\mathscr{O}_{\epsilon, (q_0, p_0)} =\{  (\overline{q},\overline{p}) \in r_\epsilon \, : \, H(\overline{q}+q_0, \overline{p}+p_0)>R\}\end{align*}  
and observe that 
\begin{align*}
S_{\epsilon, \text{D}}(\mathscr{O}_{\epsilon, (q_0, p_0)}) = \{ (\overline{q},\overline{p}) \in r_1 \, : \, H(S_{\epsilon, \text{D}}^{-1}(\overline{q}, \overline{p})+ (q_0,p_0))>R \} 
\end{align*}
where 
\begin{align*}
S_{\epsilon, \text{D}}^{-1}(\overline{q},\overline{p}) = (\epsilon^{\mathfrak{b}_1} \overline{q}^1, \ldots, \epsilon^{\mathfrak{b}_k}\overline{q}^k,  \epsilon^{1/2} \overline{p}). 
\end{align*}
Using Taylor's formula, we can write for any $(\overline{q}, \overline{p}) \in r_1$:
\begin{align}
\nonumber &H(S_{\epsilon, \text{D}}^{-1}(\overline{q}, \overline{p})+ (q_0,p_0))\\
\nonumber &= H(q_0, p_0)+ S_{\epsilon, \text{D}}^{-1}(\overline{q}, \overline{p}) \cdot\nabla H(q_0, p_0)+ \frac{1}{2}\nabla^2 H(q_0, p_0) S_{\epsilon, \text{D}}^{-1}(\overline{q}, \overline{p})\cdot S_{\epsilon, \text{D}}^{-1}(\overline{q}, \overline{p})+ R_\epsilon(\overline{q}, \overline{p})\\
\nonumber &= R+S_{\epsilon, \text{D}}^{-1}(\overline{q}, \overline{p}) \cdot\nabla H(q_0, p_0)+ \frac{1}{2}\nabla^2 H(q_0, p_0) S_{\epsilon, \text{D}}^{-1}(\overline{q}, \overline{p})\cdot S_{\epsilon, \text{D}}^{-1}(\overline{q}, \overline{p})+ R_\epsilon(\overline{q}, \overline{p}), \end{align}
where $|R_\epsilon(\overline{q}, \overline{p})|\leq C \epsilon^{3/2}$ for all $(\overline{q}, \overline{p}) \in r_1$, for some constant $C>0$.  In particular, $H(S_{\epsilon, \text{D}}^{-1}(\overline{q}, \overline{p})+ (q_0,p_0))>R$ if and only if 
\begin{align}\label{eqn:LDTaylor}
S_{\epsilon, \text{D}}^{-1}(\overline{q}, \overline{p}) \cdot\nabla H(q_0, p_0)+ \frac{1}{2}\nabla^2 H(q_0, p_0) S_{\epsilon, \text{D}}^{-1}(\overline{q}, \overline{p})\cdot S_{\epsilon, \text{D}}^{-1}(\overline{q}, \overline{p})+ R_\epsilon(\overline{q}, \overline{p})>0.  
\end{align}   

\emph{Case 1}.  $p_0^j \neq 0$ for some $j$.  Then by~\eqref{eqn:LDTaylor}, for any $\delta >0$ there exists $\epsilon_* =\epsilon_*(\delta)>0$ such that for all $\epsilon \in (0, \epsilon_*)$:
\begin{align*}
\mathscr{O}_*:= \{ (\overline{q}, \overline{p}) \in r_1 \, : \,  \overline{p}\cdot p_0 > \delta \} \subset S_{\epsilon, \text{D}}(\mathscr{O}_{\epsilon, (q_0, p_0)}).  
\end{align*}  
Also, the associated control problem is given by 
\begin{align*}
\begin{cases}
\frac{d}{dt}\overline{q}^j&= (P_\text{D}(\overline{q},\overline{p}))^j\\
\frac{d}{dt}\overline{p}^j &=\dot{f}
\end{cases}
\end{align*}
where $f\in \mathscr{H}$.  Since the values of $\overline{p}$ are explicitly controlled, $\A_{\text{D},t}(0)\cap \mathscr{O}_*\neq \emptyset$ for all $t\in (0,1]$.  By Theorem~\ref{thm:criteriareg}, the point $0$ is regular for $((\overline{q}_t, \overline{p}_t),\mathscr{O}_{q_0, p_0})$ for $(\overline{q}_t, \overline{p}_t)$ solving~\eqref{eqn:LDNP}.  Hence, for the original system~\eqref{eqn:LDG}, $(q_0, p_0)$ is regular for $((q_t, p_t), H_{>R})$ by Proposition~\ref{prop:regequiv}.

\emph{Case 2}.  $p_0 =0$.  In this case, $\mathfrak{b}_{j}=3/2$ and $P_\text{D}^j(\overline{q},\overline{p})=\overline{p}^j$ for all $j=1,2,\ldots, k$.  Furthermore,  \eqref{eqn:LDTaylor} simplifies to       
\begin{align*}
&\epsilon\frac{|\overline{p}|^2}{2} + \epsilon^{3/2} \nabla U(q_0) \overline{q} + \frac{\epsilon^{3}}{2} \nabla^2 U(q_0)\overline{q} \cdot \overline{q} + R_\epsilon(\overline{q}, \overline{p})>0.
\end{align*}
In this case, for any $\delta >0$ there exists $\epsilon_* =\epsilon_*(\delta)>0$ such that for all $\epsilon \in (0, \epsilon_*)$:
\begin{align*}
\mathscr{O}_*:= \{ (\overline{q}, \overline{p}) \in r_1 \, : \,  |\overline{p}|^2 > 2\delta \} \subset S_{\epsilon, \text{D}}(\mathscr{O}_{\epsilon, (q_0, p_0)}).  
\end{align*}  
Also, the associated control problem  is given by 
\begin{align*}
\begin{cases}
\frac{d}{dt}\overline{q}^j&= \overline{p}^j\\
\frac{d}{dt}\overline{p}^j &=\dot{f}\\
\end{cases}
\end{align*}
where $f\in \mathscr{H}$.  Again, the values of $\overline{p}$ are explicitly controlled.  Hence, $\A_{\text{D}, t}(0)\cap \mathscr{O}_*\neq \emptyset $ for all $t\in (0,1]$.  Applying Theorem~\ref{thm:criteriareg}, we see that $0$ is regular for $((\overline{q}_t, \overline{p}_t),\mathscr{O}_{q_0, 0})$ for $(\overline{q}_t, \overline{p}_t)$ solving~\eqref{eqn:LDNP}.  Hence, for the original system~\eqref{eqn:LDG}, $(q_0, 0)$ is regular for $((q_t, p_t), H_{>R})$ by Proposition~\ref{prop:regequiv}. 
 \end{proof}

\bibliographystyle{plain}
\bibliography{LIL_PDE}

\begin{thebibliography}{10}

\bibitem{AS_04}
A.A. Agrachev and Y.L. Sachkov.
\newblock {\em Control Theory from the Geometric Viewpoint}, volume~2.
\newblock Springer Science \& Business Media, 2004.

\bibitem{AS_06}
A.A. Agrachev and A.V. Sarychev.
\newblock Controllability of {2D} {E}uler and {N}avier-{S}tokes equations by
  degenerate forcing.
\newblock {\em Comm. Math. Phys.}, 265(3):673--697, 2006.

\bibitem{BL_22}
J.~Bedrossian and K.~Liss.
\newblock Stationary measures for stochastic differential equations with
  degenerate damping.
\newblock {\em arXiv preprint arXiv:2206.02240}, 2022.

\bibitem{Bell_12}
D.R. Bell.
\newblock {\em The {M}alliavin {C}alculus}.
\newblock Courier Corporation, 2012.

\bibitem{Blumenthal1957}
R.~M. Blumenthal.
\newblock An extended {M}arkov property.
\newblock {\em Trans. Am. Math. Soc.}, 85(1):52--72, 1957.

\bibitem{BG_07}
R.M. Blumenthal and R.K. Getoor.
\newblock {\em Markov {P}rocesses and {P}otential {T}heory}.
\newblock Courier Corporation, 2007.

\bibitem{CFH_22}
M.~Carfagnini, J.~F{\"o}ldes, and D.P. Herzog.
\newblock A functional law of the iterated logarithm for weakly hypoelliptic
  diffusions at time zero.
\newblock {\em Stoch. Proc. Appl.}, 149:188--223, 2022.

\bibitem{FH_23}
J.~F{\"o}ldes and D.~P. Herzog.
\newblock The method of stochastic characteristics for linear second-order
  hypoelliptic equations.
\newblock {\em Probab. Surv.}, 20:113--169, 2023.

\bibitem{GS_90}
N.~Garofalo and F.~Segala.
\newblock Estimates of the fundamental solution and {W}iener's criterion for
  the heat equation on the {H}eisenberg group.
\newblock {\em Indiana Univ. Math. J.}, 39(4):1155--1196, 1990.

\bibitem{GHHM_18}
N.E. Glatt-Holtz, D.P. Herzog, and J.C. Mattingly.
\newblock Scaling and saturation in infinite-dimensional control problems with
  applications to stochastic partial differential equations.
\newblock {\em Ann. PDE}, 4:1--103, 2018.

\bibitem{HM_15}
D.P. Herzog and J.C. Mattingly.
\newblock A practical criterion for positivity of transition densities.
\newblock {\em Nonlinearity}, 28(8):2823, 2015.

\bibitem{Hor_67}
L.~H{\"o}rmander.
\newblock Hypoelliptic second order differential equations.
\newblock {\em Acta Math.}, 119(1):147--171, 1967.

\bibitem{Jur_97}
V.~Jurdjevic.
\newblock {\em Geometric {C}ontrol {T}heory}.
\newblock Cambridge {U}niversity {P}ress, 1997.

\bibitem{JK_85}
V.~Jurdjevic and I.~Kupka.
\newblock Polynomial control systems.
\newblock {\em Math. Ann.}, 272(3):361--368, 1985.

\bibitem{Kog_17}
A.E. Kogoj.
\newblock On the {D}irichlet problem for hypoelliptic evolution equations:
  {P}erron--{W}iener solution and a cone-type criterion.
\newblock {\em J. Differ. Equations}, 262(3):1524--1539, 2017.

\bibitem{KS_84}
S.~Kusuoka and D.~Stroock.
\newblock Applications of the {M}alliavin {C}alculus, {P}art {I}.
\newblock In {\em North-Holland Mathematical Library}, volume~32, pages
  271--306. Elsevier, 1984.

\bibitem{KS_85}
S.~Kusuoka and D.~Stroock.
\newblock Applications of the {M}alliavin {C}alculus, {P}art {II}.
\newblock {\em J. Fac. Sci. Univ. Tokyo}, 32(1-76):18, 1985.

\bibitem{KS_87}
S.~Kusuoka and D.~Stroock.
\newblock Applications of the {M}alliavin {C}alculus, {P}art {III}.
\newblock {\em J. Fac. Sci. Univ. Tokyo Sect IA Math}, 34:391--442, 1987.

\bibitem{Lac_97ii}
A.~Lachal.
\newblock Local asymptotic classes for the successive primitives of brownian
  motion.
\newblock {\em Ann. of Probab.}, pages 1712--1734, 1997.

\bibitem{Lac_97i}
A.~Lachal.
\newblock Regular points for the successive primitives of brownian motion.
\newblock {\em J. Math. Kyoto U.}, 37(1):99--119, 1997.

\bibitem{LTU_17}
E.~Lanconelli, G.~Tralli, and F.~Uguzzoni.
\newblock {W}iener-type tests from a two-sided {G}aussian bound.
\newblock {\em Ann. Mat. Pura Appl.}, 196(1):217--244, 2017.

\bibitem{sabra_98}
V.S. {L'vov}, E.~Podivilov, A.~Pomyalov, I.~Procaccia, and D.~Vandembroucq.
\newblock Improved shell model of turbulence.
\newblock {\em Phys. Rev. E}, 58(2):1811, 1998.

\bibitem{MP_06}
J.C. Mattingly and {\'E}.~Pardoux.
\newblock Malliavin calculus for the stochastic {2D} {N}avier--{S}tokes
  equation.
\newblock {\em Comm. Pure Appl. Math.}, 59(12):1742--1790, 2006.

\bibitem{NS_87}
P.~Negrini and V.~Scornazzani.
\newblock Wiener criterion for a class of degenerate elliptic operators.
\newblock {\em J. Diff. Equations}, 66(2):151--164, 1987.

\bibitem{NM_24}
V.~Nersesyan and M.~Rissel.
\newblock Global controllability of {B}oussinesq flows by using only a
  temperature control.
\newblock {\em arXiv preprint:2404.09903}, 2024.

\bibitem{Norris_06}
J.~Norris.
\newblock Simplified {M}alliavin {C}alculus.
\newblock In {\em S{\'e}minaire de Probabilit{\'e}s XX 1984/85: Proceedings},
  pages 101--130. Springer, 2006.

\bibitem{Nualart_09}
D.~Nualart.
\newblock {\em Malliavin {C}alculus and its {A}pplications}.
\newblock Number 110. American Mathematical Soc., 2009.

\bibitem{OK_13}
B.~Oksendal.
\newblock {\em Stochastic {D}ifferential {E}quations: {A}n {I}ntroduction with
  {A}pplications}.
\newblock Springer Science \& Business Media, 2013.

\bibitem{PS_72}
S.C. Port and C.J. Stone.
\newblock Classical potential theory and brownian motion.
\newblock In {\em Proc. Sixth Berkeley Symp. on Math. Stat. and Prob., Univ. of
  Cal. Press}, volume~3, pages 143--176, 1972.

\bibitem{SV_72ii}
D.W. Stroock and S.R.S. Varadhan.
\newblock On degenerate elliptic-parabolic operators of second order and their
  associated diffusions.
\newblock {\em Comm. Pure Appl. Math.}, 25(6):651--713, 1972.

\bibitem{SV_72i}
D.W. Stroock and S.R.S. Varadhan.
\newblock On the support of diffusion processes with applications to the strong
  maximum principle.
\newblock In {\em Proceedings of the Sixth Berkeley Symposium on Mathematical
  Statistics and Probability (Univ. California, Berkeley, Calif., 1970/1971)},
  volume~3, pages 333--359, 1972.

\end{thebibliography}

\end{document}